\documentclass[11pt]{amsart}
\usepackage[letterpaper,margin=1in]{geometry}

\pdfoutput=1

\usepackage[utf8]{inputenc}
\usepackage{epsfig}
\usepackage{amssymb}
\usepackage{amsmath}
\usepackage{amsthm}
\usepackage{graphicx}
\usepackage{subcaption}
\usepackage{enumitem}
\usepackage{epstopdf}
\usepackage{color}

\def\u{{\text {\textbf u}}}
\def\v{{\text {\textbf v}}}
\def\e{{\text {\textbf e}}}
\def\1{{\text {\textbf 1}}}
\def\f{{\text {\textbf f}}}
\def\d{{\text {\textbf d}}}

\makeatletter

\newtheorem{theorem}{Theorem}[section]

\newtheorem{lemma}[theorem]{Lemma}
\newtheorem{remark}[theorem]{Remark}

\newtheorem{proof*}[theorem]{Proof}
\newtheorem{algorithm}[theorem]{Algorithm}

\begin{document}

\title[ASMG Method Based on ASCA for Graph Laplacian]{Auxiliary Space Multigrid Method Based on Additive
Schur Complement Approximation for Graph Laplacian}
\author[Maria Lymbery]{Maria Lymbery}
\address{Faculty of Mathematics,
University of Duisburg-Essen,
Thea-Leymann-Str.~9, 45127 Essen, Germany}
\email{maria.lymbery@uni-due.de}

\maketitle


\begin{abstract}
This research studies the application of the auxiliary space multigrid method (ASMG) that is based on additive Schur complement 
approximation (ASCA) to graph Laplacian matrices arising from general graphs. A major predicament when considering algebraic multigrid (AMG) methods 
on such graphs is the choice of a general coarsening strategy which has to be both cheap and effective. Such a strategy has been incorporated in the presented 
approach which in addition has several advantages. First, it is purely algebraic in its construction which makes the algorithm easy to implement. Furthermore, the approach requires 
no limitation on the graph's structure and itself can be adjusted to the particular problem. Last but not least, its computational complexity can be easily analysed. A demonstrative set of 
numerical experiments is presented.
\end{abstract}

\section{Introduction}

Laplacian matrices of graphs have a wide range of applications including, but not limited to, machine learning, clustering in images, data mining, 
see e.g.~\cite{B-06,J-03, L-07, Q-12, Z-03}. 
Additionally, they also play an important role in finite element and finite difference discretizations of elliptic partial differential equations (PDEs) describing various 
physical phenomena. A structured and 
detailed overview about their significance could be found, for example, in the work of Spielman~\cite{S-10}. 

As would be expected, the design and development of fast solvers for corresponding linear systems has been the focus of considerable research 
where multilevel/multigrid algorithms have been of particular interest.
Notable results include aggregation-based, see e.g.~\cite{L-12,B-13,B-14,N-16}, and disaggregation-based AMG preconditioners, see e.g.~\cite{D-15}, 
accelerated,~\cite{DA-15}, and  combinatorial multigrid and multilevel preconditioners,
see~\cite{K-09, K-07}. 
In~\cite{L-12} Livne and Brandt have introduced the lean algebraic multigrid method based on aggregation motivated by a new vertex proximity measure 
and simple piecewise constant prolongators coupled with an energy correction procedure applied to coarse-level systems. 
Brannick et al. have presented 
estimates of the convergence rate and complexity of an AMG preconditioner based on piecewise constant coarse vector spaces applied to the graph Laplacian, 
see~\cite{B-13}. 
The approach proposed in~\cite{B-14} combines aggressive coarsening based on aggregation 
with a polynomial smoother with sufficiently large degree to solve Laplacian systems arising from the standard linear
finite element discretization of the scalar Poisson problem. 
More recently, Napov and Notay, see~\cite{N-16}, have developed an aggregation based multigrid method that relies on the recursive static elimination
of the vertices of degree $1$ combined with a new Degree-aware Rooted Aggregation (DRA) algorithm. 
The adaptive algebraic multigrid method proposed by D'Ambra and Vassilevski for solving Graph Laplacians, see~\cite{D-15}, relies on a disaggregation technique where the few high degree nodes 
are broken into multiple smaller degree nodes. 
The authors of~\cite{DA-15} offer multigrid type techniques that combine ad hoc coarser-grid operators with iterative techniques used as smoothers 
for the numerical solution of graph Laplacian operators. 
The results presented in~\cite{K-09, K-07} suggest an approach to construct multigrid-like solvers based on support theory principles. 
Graph Laplacian preconditioners have been studied in~\cite{S-10} and~\cite{V-12} where graph sparsification techniques have been used to maintain reasonable computational 
complexity in the case of general large graphs. 

The solver advocated here results from the interplay between graph and multigrid theory which makes it universal from the view point of applicability and 
construction. The utilized multigrid method is based on the additive Schur complement approximation (ASCA), see~\cite{K-12}, used to construct 
coarse spaces and auxiliary-space correction, see~\cite{K-15,H-07,X-96}, which replaces the standard multigrid coarse-grid correction. Finding maximal independent 
subsets (MIS) in graphs is essential for the proposed auxiliary space multigrid algorithm as they specify not only the coarse-fine splitting of the degrees of freedom, 
as e.g. in~\cite{A-10,M-01}, but also the construction of ASCA. 

The rest of the paper is organized as follows. In Section~\ref{pr} the graph Laplacian model problem is presented, the basic notations are introduced and a procedure generating the 
building components of the multigrid method is described. The next section, Section~\ref{ASMG}, includes the definitions of the ASCA and the auxiliary space 
multigrid algorithm. 
The complexity of the proposed algorithm has been discussed in Section~\ref{sec:complexity}.
A demonstrative set of experiments is presented in Section~\ref{ne}. 
Finally, some concluding remarks have been made.

\section{Problem formulation}\label{pr}
\subsection{Description and assumptions}
 Consider the undirected graph $K$ comprising 
 the set of vertices, also known as nodes, $\mathcal{V}$ together with the set of edges $\mathcal{E}$ (which are $2$-element subsets of $\mathcal{V}$)
$$
K = (\mathcal{V}, \mathcal{E}), 
$$
where
$\vert \mathcal{V}\vert=n$ and $\vert \mathcal{E}\vert=m$. 
In what follows, we further assume that $K$ is an unweighted and connected graph. 
Our aim is to solve the linear system 
\begin{equation}\label{eq:pr_form}
A\u=\f
\end{equation}
where $\f \in\text{Range}(A)$ while $A$ is the Laplacian matrix related to the graph $K$ as
$$
(A)_{ij}=\begin{cases}
\;\, d_i \qquad\qquad  i=j;\\[2.2pt]
 -1 \qquad\qquad  i\neq j, \, (i,j)\in\mathcal{E}; \\[2.2pt]
\;\;\;0 \qquad\qquad i\neq j, \, (i,j)\notin\mathcal{E}.
\end{cases}
$$
Here $d_i$ denotes the degree of the $i$-th vertex.

Obviously, $A$ is a symmetric, positive semi-definite (SPSD), singular M-matrix whose kernel is spanned by the constant vector $\1$.
One way to deal with the semi-definiteness of the problem is to apply a rank-$1$ update of the matrix, thereby 
obtaining an equivalent SPD problem, see~\cite{D-15} for more details. Another remedy is the application of 
a deflated version of the conjugate gradient (CG) method to maintain the residuals orthogonal to the kernel when 
solving iteratively the linear system, see~\cite{SYEG-06}.

\begin{remark}
Problem~\eqref{eq:pr_form} can be equivalently presented in the variational form
$$
(A\u,\v) = (\f,\v), \qquad \forall \v \in \mathbb{R}^n, 
$$
where
$$
(A\u,\v)=\sum_{e=(i,j)\in \mathcal{E}}(u_i-u_j)(v_i-v_j), \qquad (\f,\v)=\sum_{i\in\mathcal{V}}f_i v_i \qquad \text{and}\;\; (\f,\1)=0.
$$
\end{remark}

\subsection{Preliminaries and notation}\label{sec:prel_not}
Consider subgraphs $K_G=(\mathcal{V}_G,\mathcal{E}_G)$ and $K_F=(\mathcal{V}_F,\mathcal{E}_F)$ of $K$ such that
\begin{subequations}
\begin{alignat}{4}
\forall e\in \mathcal{E}&\;\; \text{there exists}\;\; K_F\in \mathcal{F}=\{K_F\}: e\in \mathcal{E}_F \\
\forall K_F\in \mathcal{F}&\;\; \text{there exists}\;\; K_G\in \mathcal{G}=\{K_G\}: K_F\subset K_G
\end{alignat}\label{eq:defs}
\end{subequations}

Similarly, as for the notation introduced in~\cite{K-15} we refer to the subgraphs $K_F$ as \textit{structure subgraphs}, 
and to $K_G$ as \textit{macrostructure graphs}. The Laplacian matrix $A$ then can be assembled from the 
local matrices $A_F$ and $A_G$, i.e.,

$$
A = \sum_{K_F \in \mathcal{F}} R_F^T A_F R_F
$$
and
$$
A = \sum_{K_G \in \mathcal{G}} R_G^T A_G R_G ,
$$
where $R_F^T$ and $R_G^T$ are the standard inclusion operators.

\begin{remark}
It is important to note that when the covering of the graph $K$ by structure subgraphs $\mathcal{F}$ is such that no two structure subgraphs share an edge, 
then the matrices $A_F$ are the standard Laplacian matrices associated with the subgraphs $K_F$. 
If in addition there are no two macrostructures sharing a structure subgraph, the same is applicable to the matrices $A_G$ associated with the subgraphs 
$K_G$.   
\end{remark}

The macrostructure matrices $A_G$ themselves can be assembled from the structure matrices $A_F$ as given by
\begin{equation}\label{s_ms}
A_G = \sum_{K_F \subset K_G} \sigma_{F,G} R_{K_G \mapsto K_F}^T A_F R_{K_G \mapsto K_F}.
\end{equation}
Here, the scaling factors $\sigma_{F,G}$ provide a partition of unity:
$$
\sum_{K_G \supset K_F} \sigma_{F,G} = 1 \quad \forall K_F \in \mathcal{F}.
$$

Our next aim is to present a procedure for generating structure and macrostructure matrices which are the building blocks of the auxiliary space multigrid method: 

\begin{itemize}
\item[Step I:] 
Starting with a given graph we find a maximal independent set of nodes. For convenience let us denote it by 
$\mathcal{V}^1$. Each of the nodes in $\mathcal{V}^1$ then 
serves the purposes of a "focus" of a structure subgraph and would identify and define this subgraph. 

Figure~\ref{Fig1} illustrates one particular graph with nodes represented by circles and edges by lines. On 
the left subfigure all nodes are in black whereas on the right only the nodes belonging to a maximal independent 
set remain in black. 

\begin{figure}[ht]
\begin{tabular}{@{}c@{\hspace{.1cm}}c@{}}
\includegraphics[width=8.5cm]{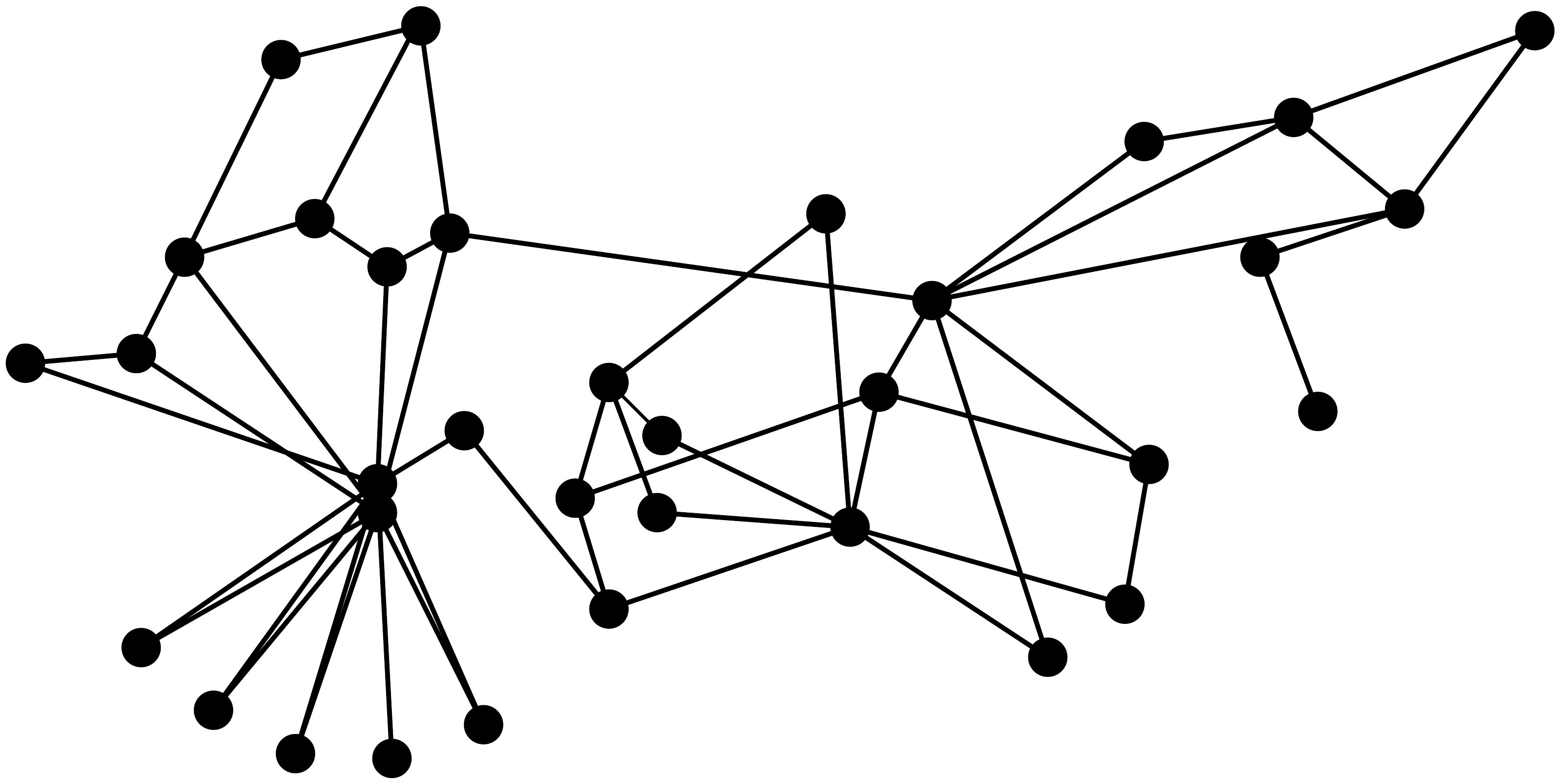} & \includegraphics[width=8.5cm]{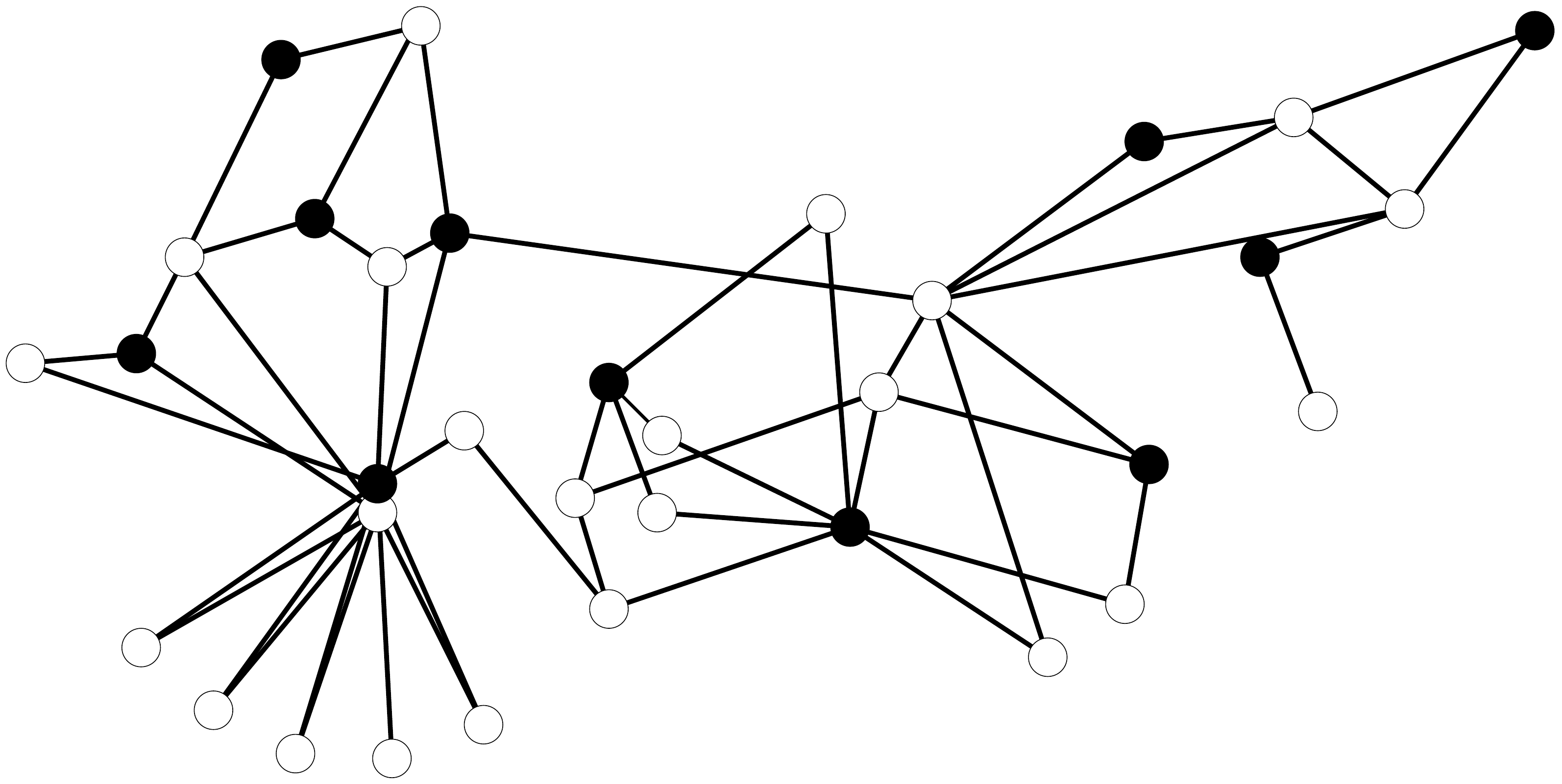}
\end{tabular}
\caption{}
\label{Fig1}
\end{figure}

\item[Step II:] The set of nodes for a structure subgraph consists of a focus node and all nodes that are at a graph distance $1$ or $2$ to it, which might 
include, of course, also other foci nodes. The set of edges is then formed as all edges from the original graph that connect the nodes belonging to the 
structure subgraph.  

Figures~\ref{Fig2}--\ref{Fig7} show the structure subgraphs for our example. We have found a maximal independent set consisting 
of $11$ nodes and therefore we have formed $11$ subgraphs.

\begin{figure}[h!]
\begin{tabular}{@{}c@{\hspace{.1cm}}c@{}}
\includegraphics[width=8.5cm]{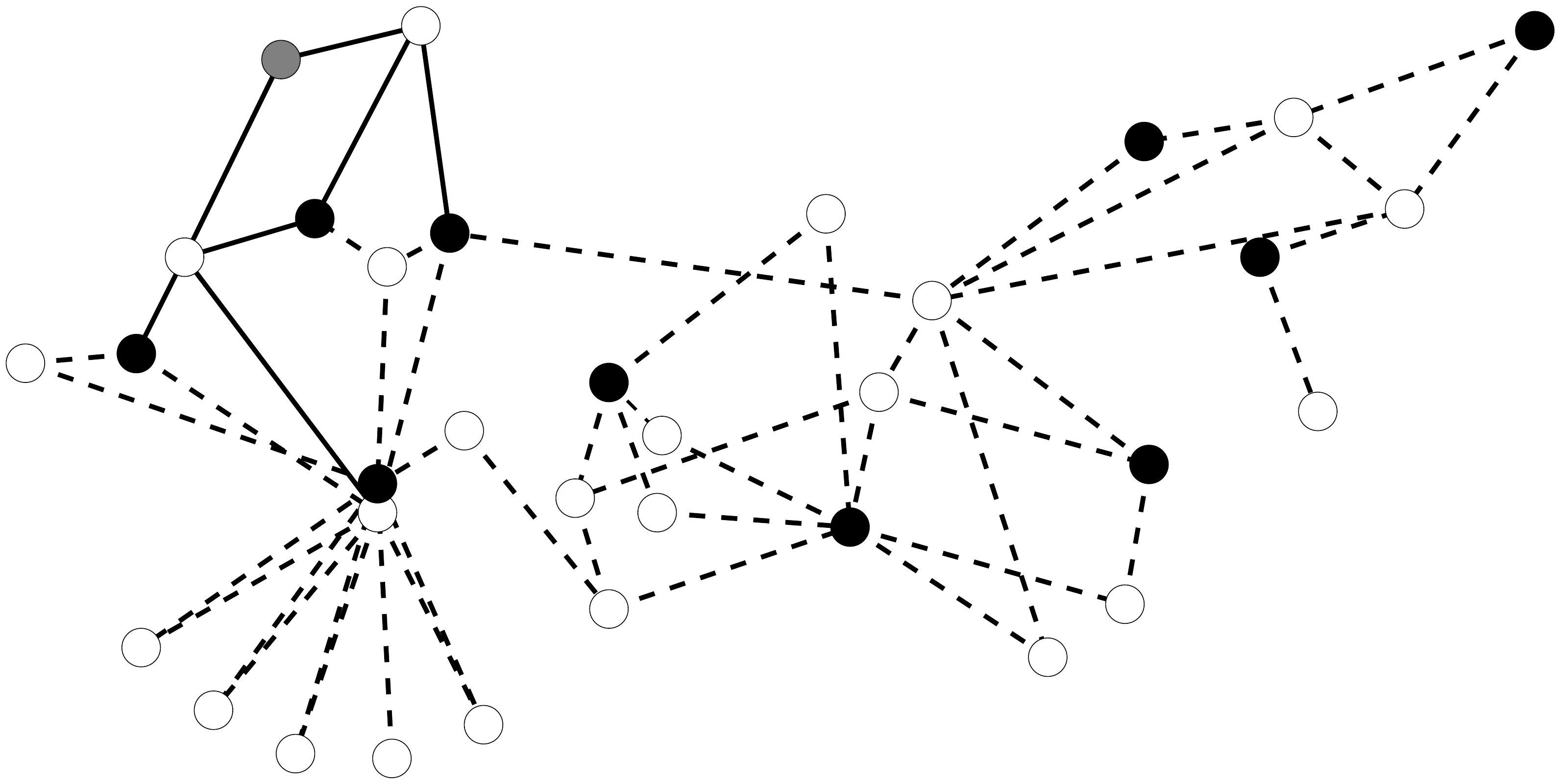}
&
\includegraphics[width=8.5cm]{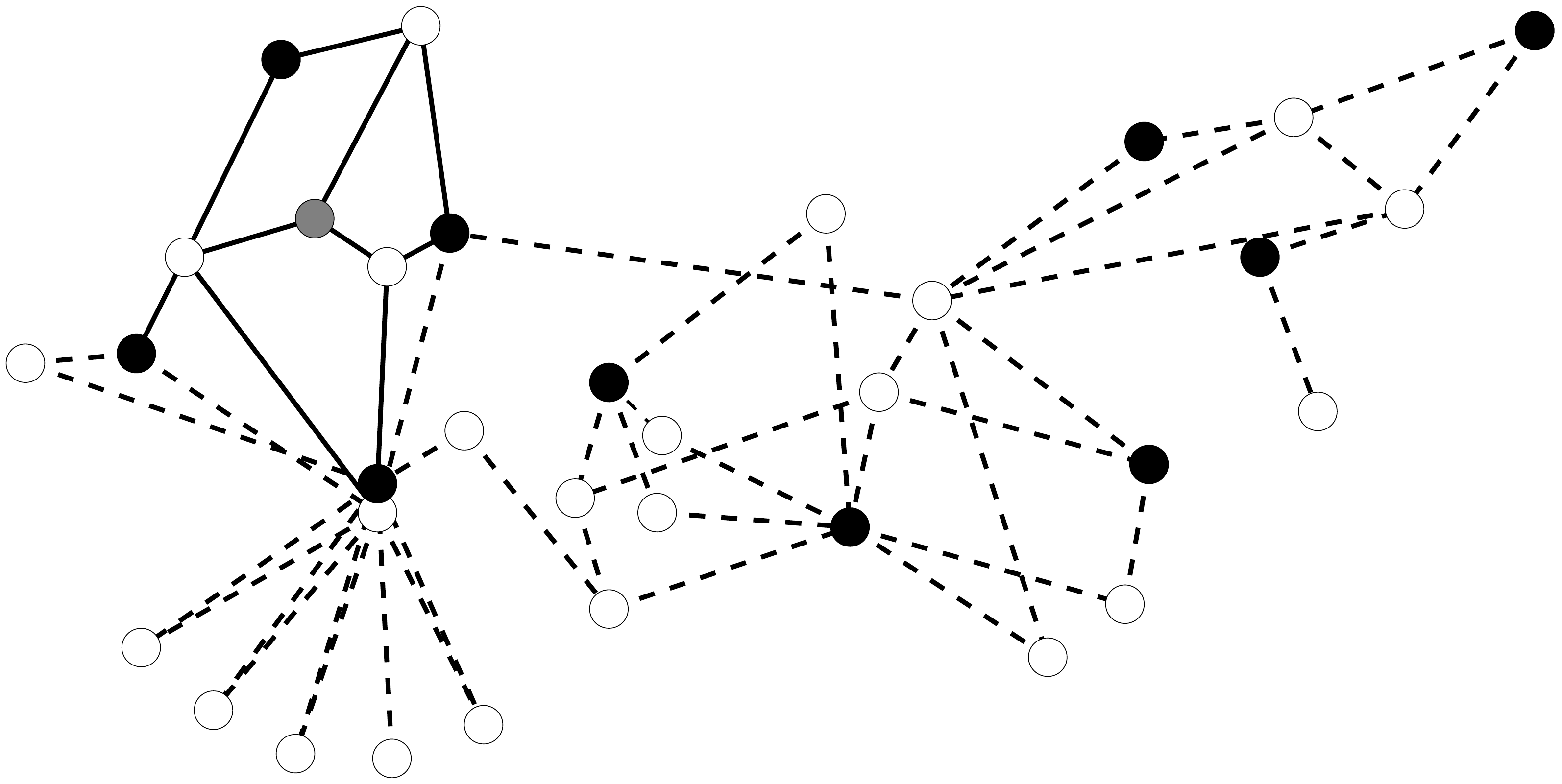}
\end{tabular}
\caption{}
\label{Fig2}
\end{figure}

\begin{figure}[h!]
\begin{tabular}{@{}c@{\hspace{.1cm}}c@{}}
\includegraphics[width=8.5cm]{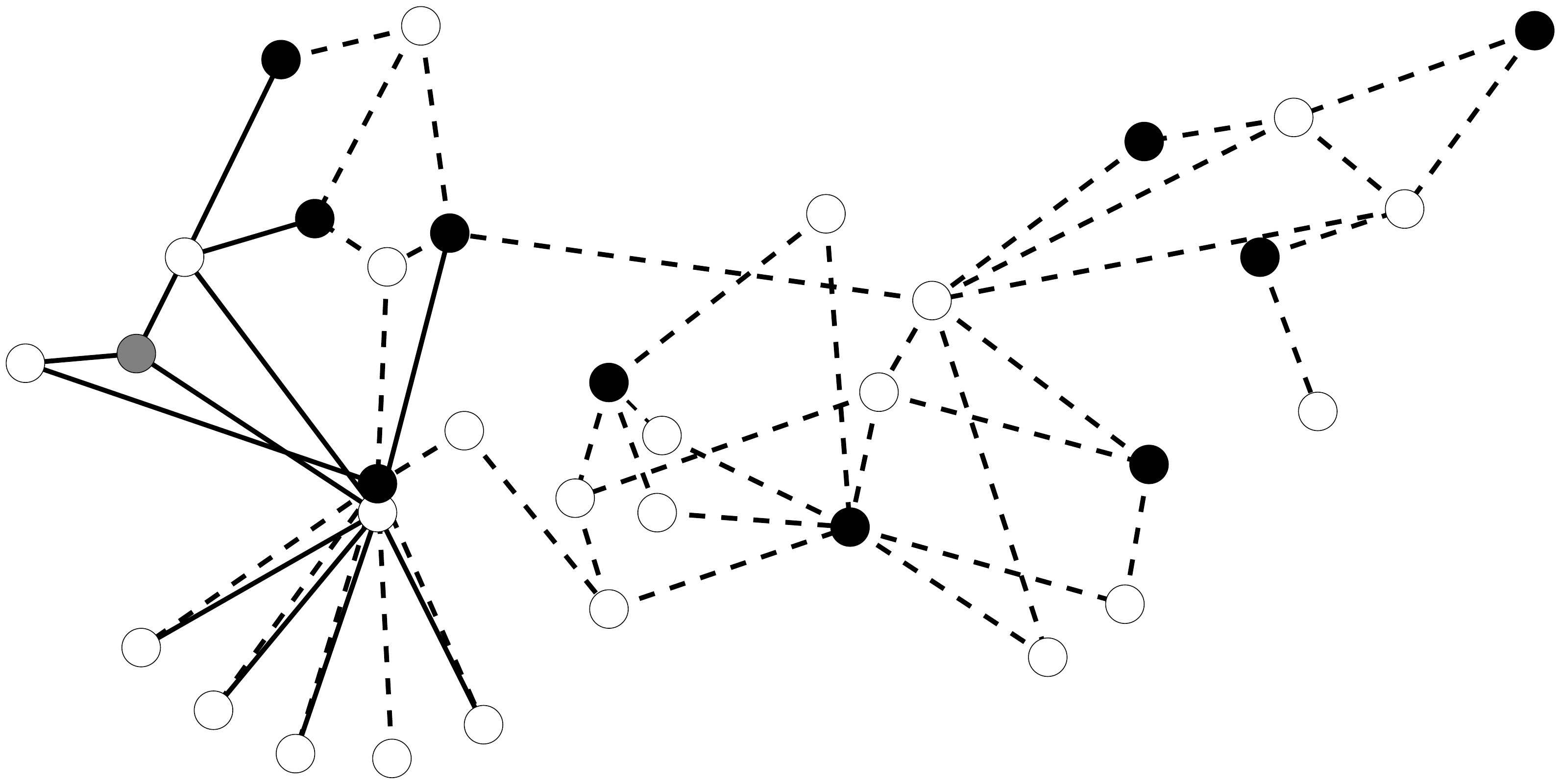}
&
\includegraphics[width=8.5cm]{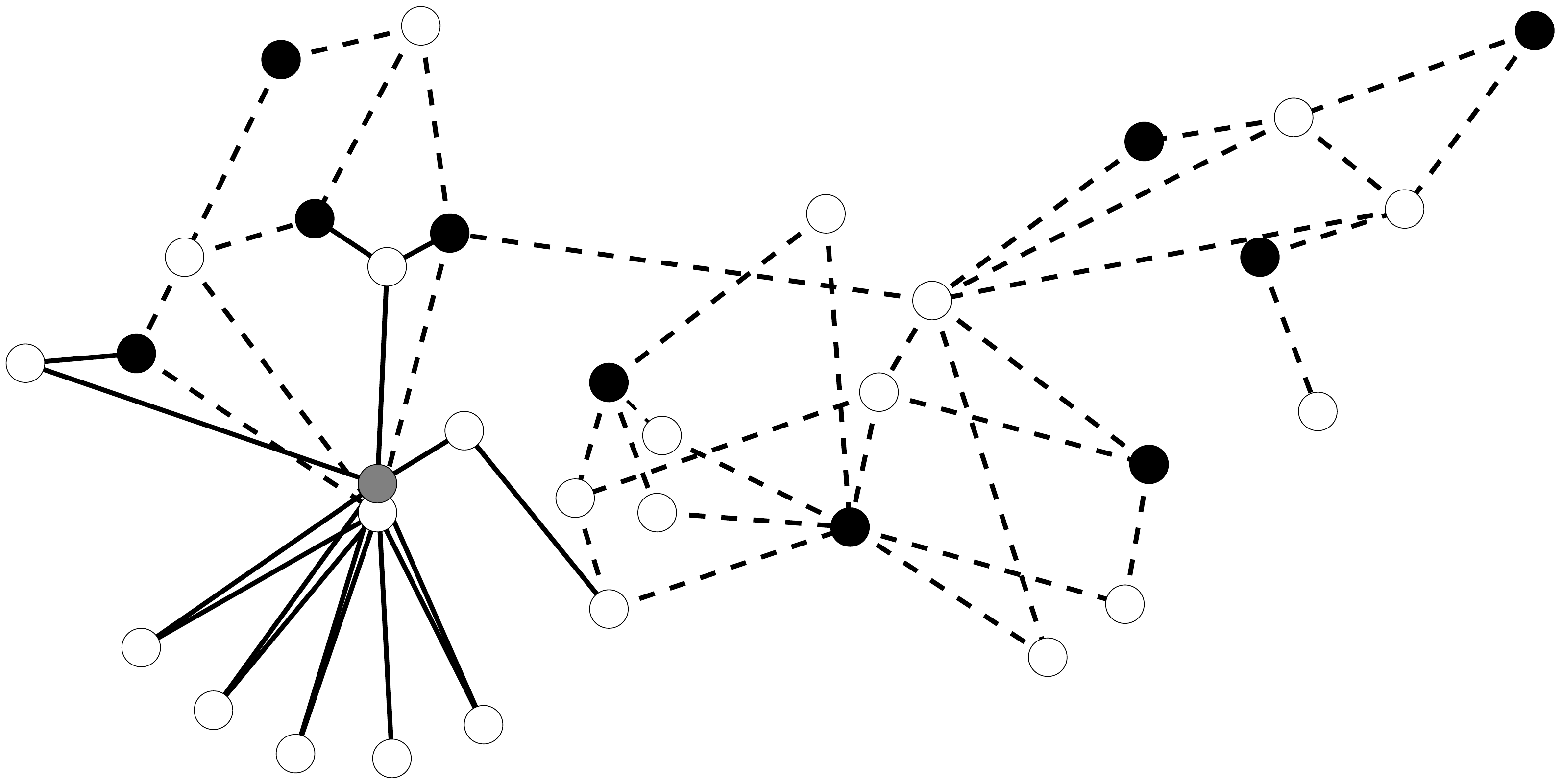}
\end{tabular}
\caption{}
\label{Fig3}
\end{figure}

\begin{figure}[h!]
\begin{tabular}{@{}c@{\hspace{.1cm}}c@{}}
\includegraphics[width=8.5cm]{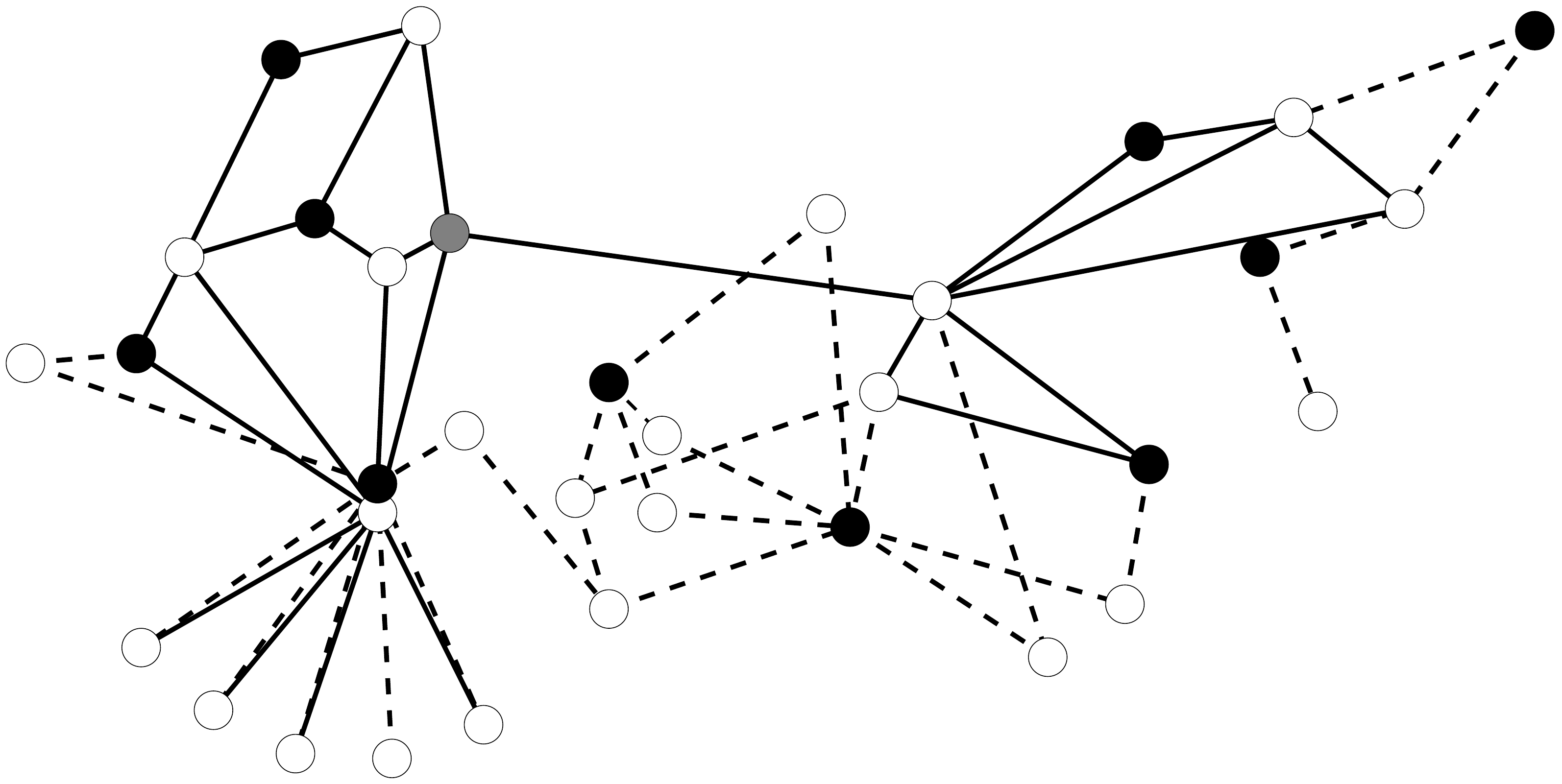}
&
\includegraphics[width=8.5cm]{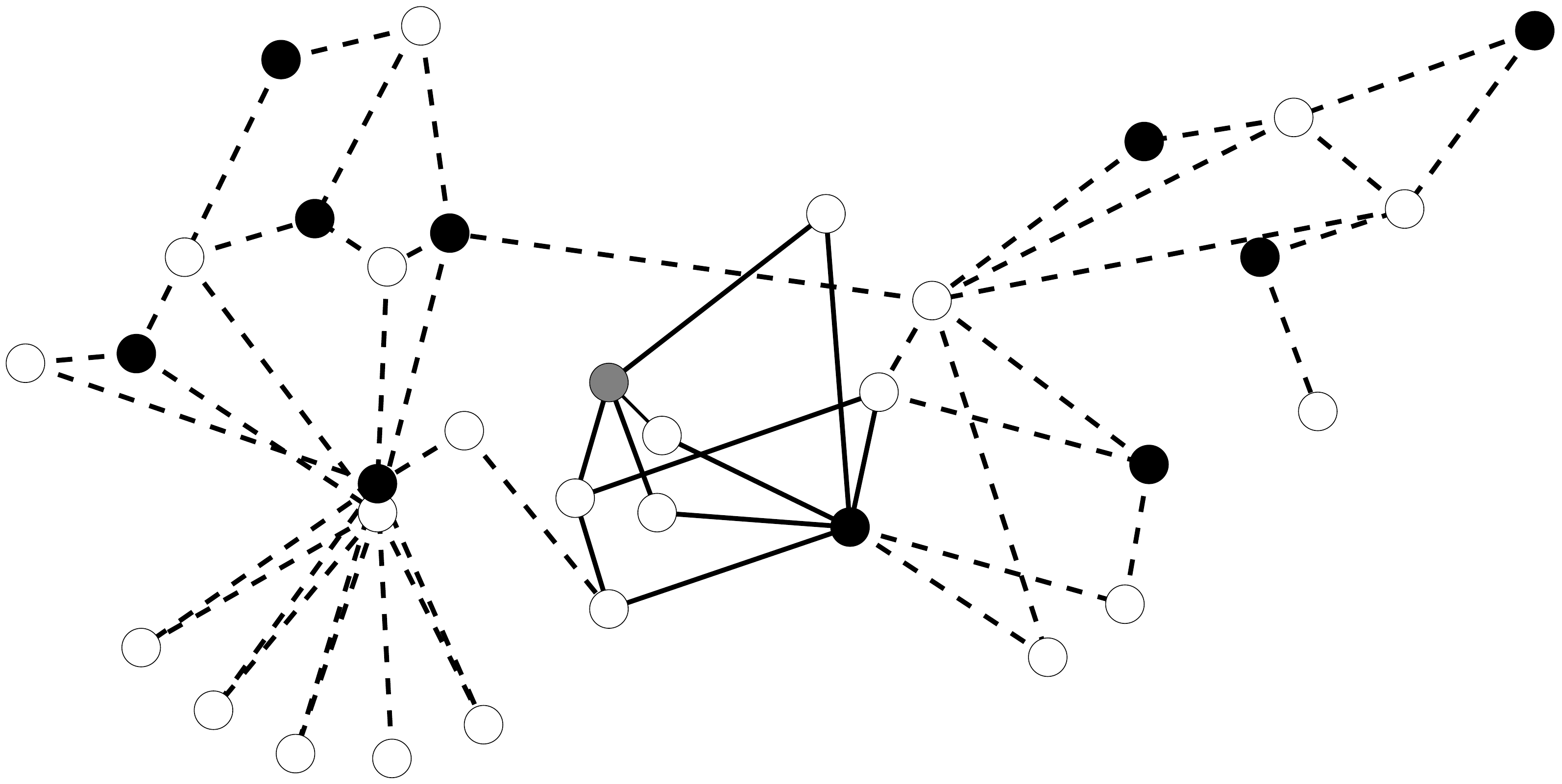}
\end{tabular}
\caption{}
\label{Fig4}
\end{figure}

\begin{figure}[h!]
\begin{tabular}{@{}c@{\hspace{.1cm}}c@{}}
\includegraphics[width=8.5cm]{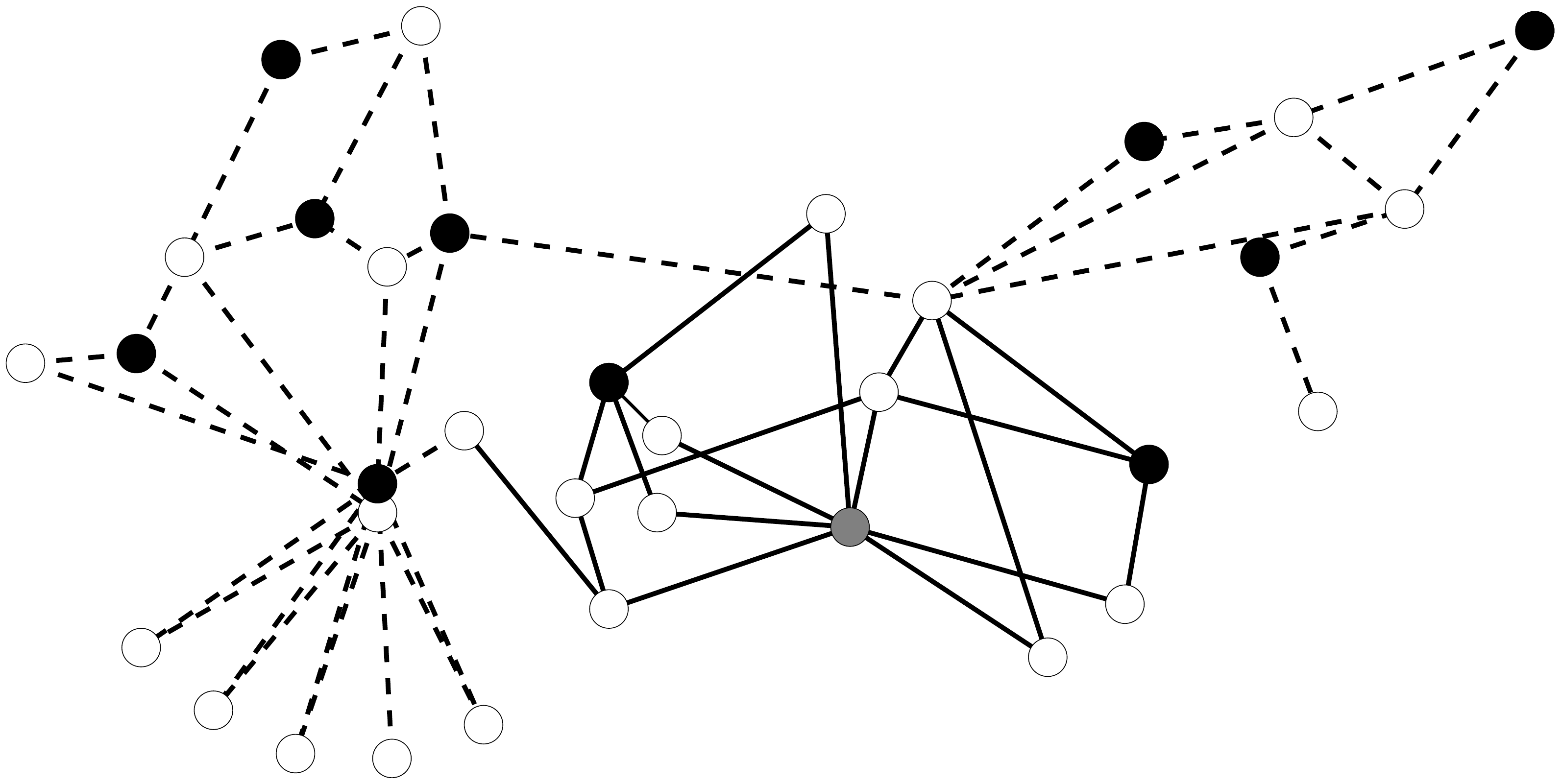}
&
\includegraphics[width=8.5cm]{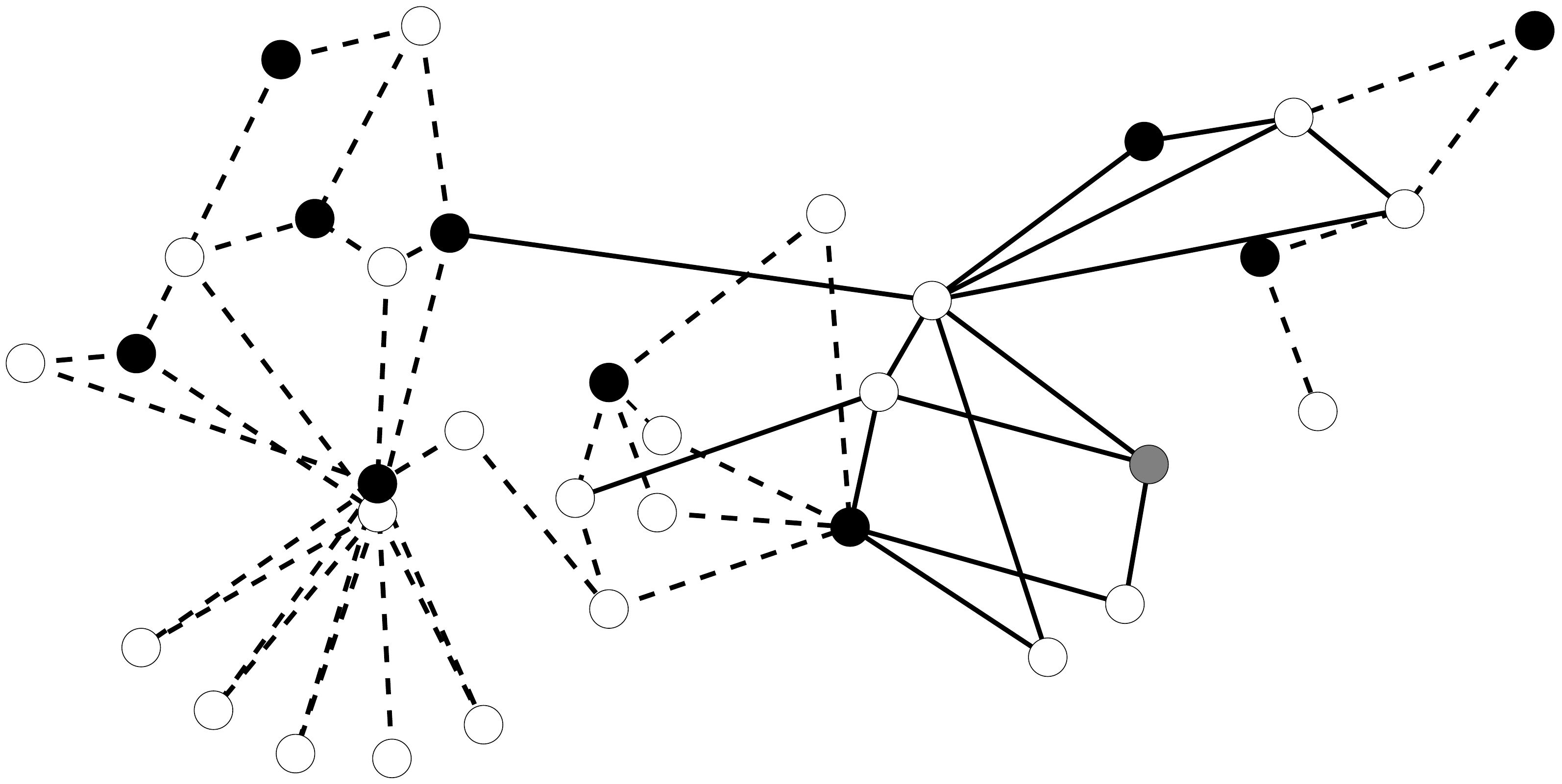}
\end{tabular}
\caption{}
\label{Fig5}
\end{figure}

\begin{figure}[h!]
\begin{tabular}{@{}c@{\hspace{.1cm}}c@{}}
\includegraphics[width=8.5cm]{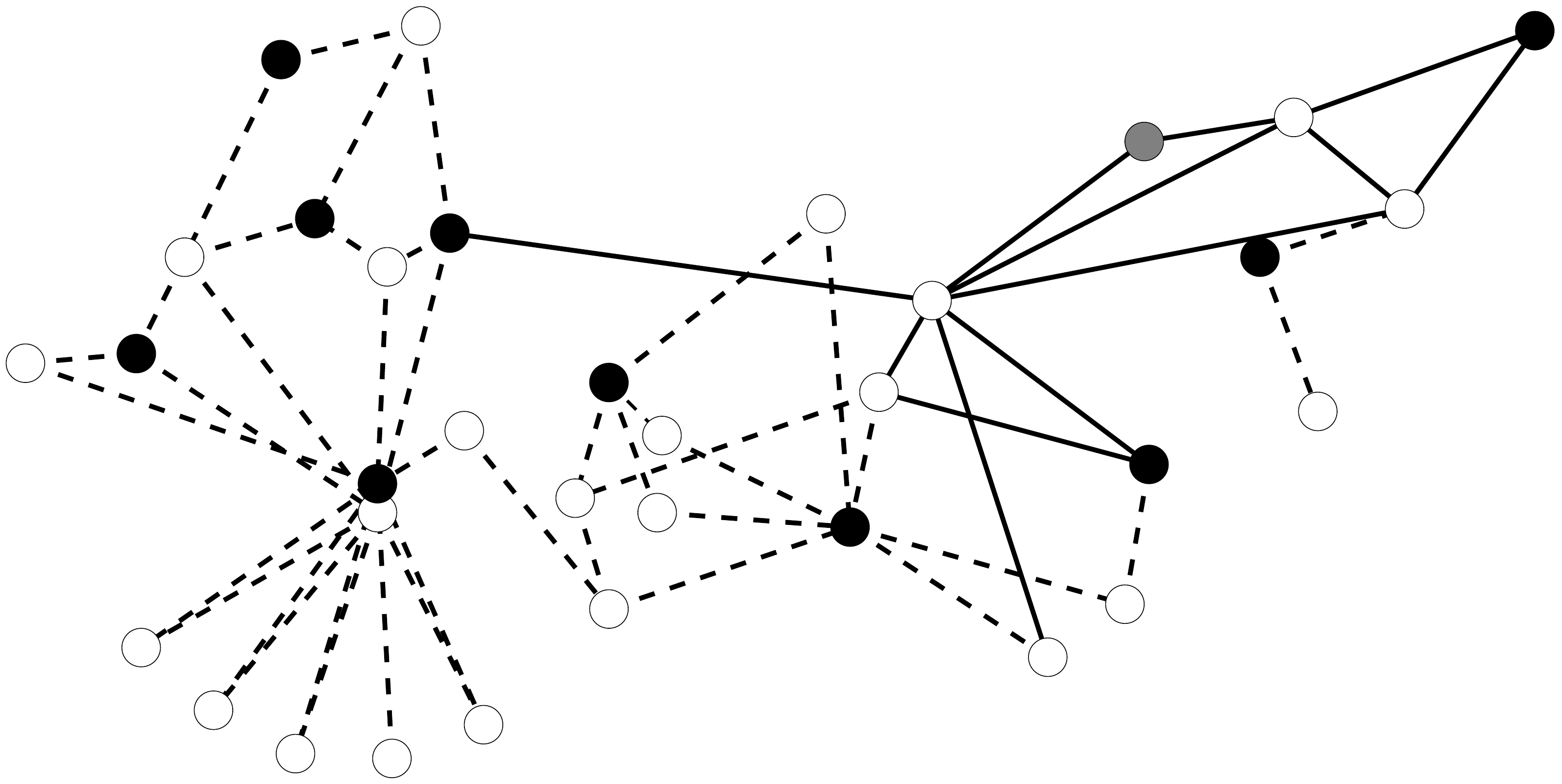}
&
\includegraphics[width=8.5cm]{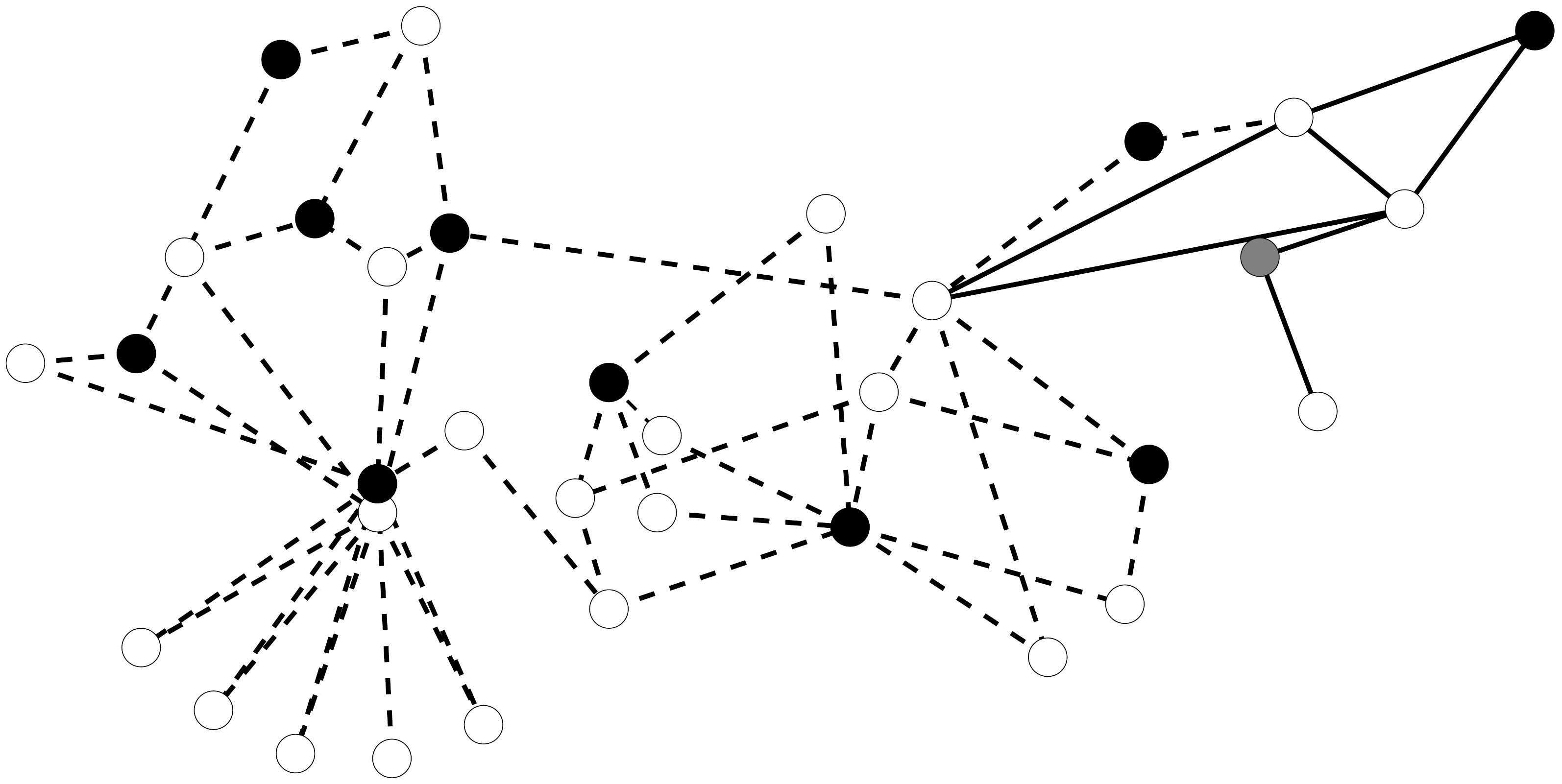}
\end{tabular}
\caption{}
\label{Fig6}
\end{figure}

\begin{figure}[h!]
\centering
\includegraphics[width=8.5cm]{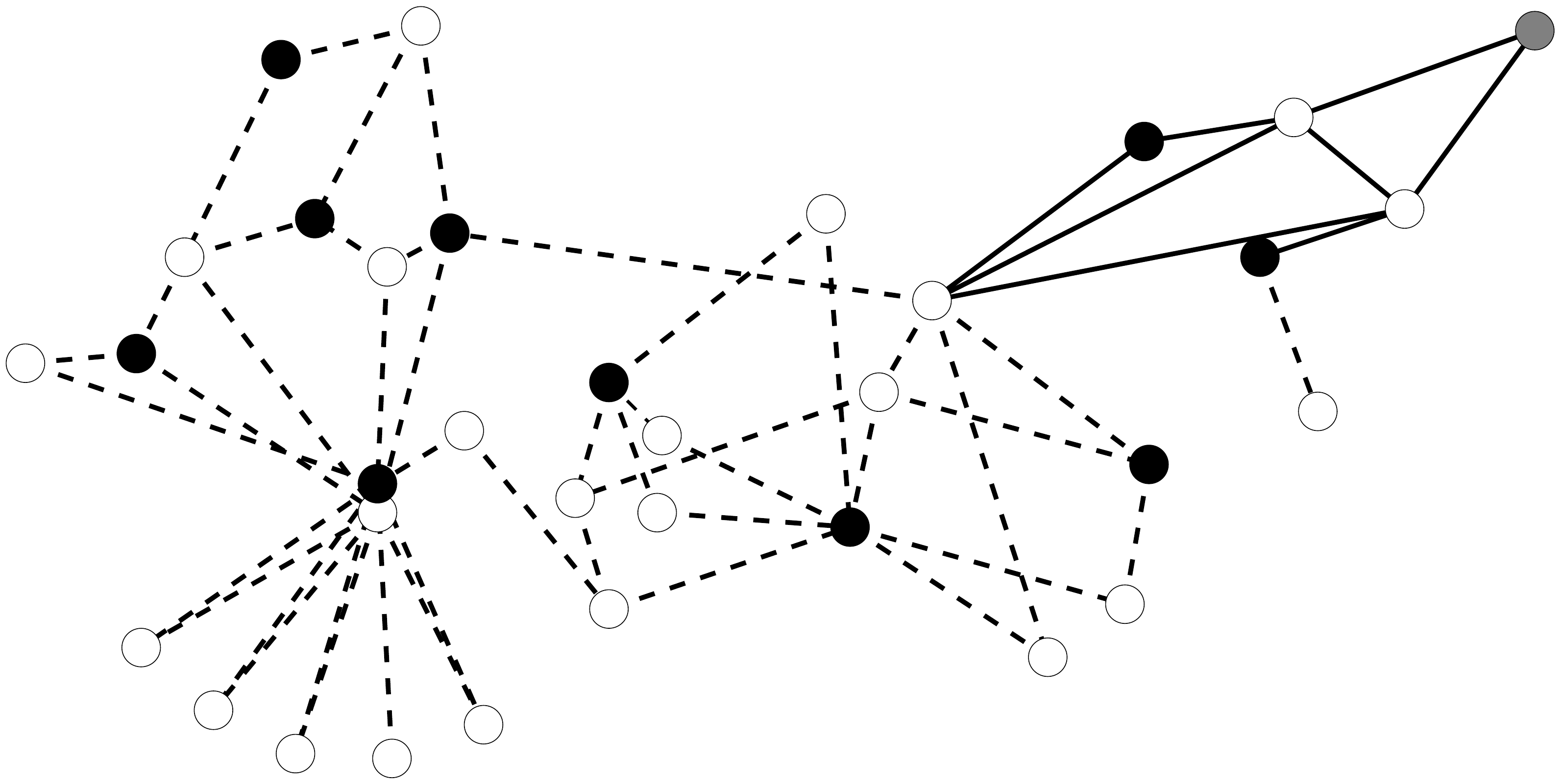}
\caption{}
\label{Fig7}
\end{figure}

\item[Step III:] We construct a "coarser subgraph" $K^{1}$ from the original graph. Its nodes are the nodes from $\mathcal{V}^1$ 
and two nodes are connected by an edge in $K^{1}$ if in the original graph they are at a graph distance $2$. Note that by 
definition this is the smallest graph distance at which any nodes from $\mathcal{V}^1$ could be in $K$.

\item[Step IV:] We find a maximal independent set of nodes in the "coarser subgraph" and denote it by 
$\mathcal{V}^{2}$. Each of the nodes in $\mathcal{V}^{2}$ will play the role of a "focus" of a macrostructure subgraph and 
would specify this subgraph. 

Step III and Step IV of the procedure are represented on Figure~\ref{Fig8} in a similar way as in Figure~\ref{Fig1}. This time the maximal independent 
subset consists of the $5$ nodes depicted in black on the right subfigure.

\begin{figure}[h!]
\begin{tabular}{@{}c@{\hspace{.1cm}}c@{}}
\includegraphics[width=8.5cm]{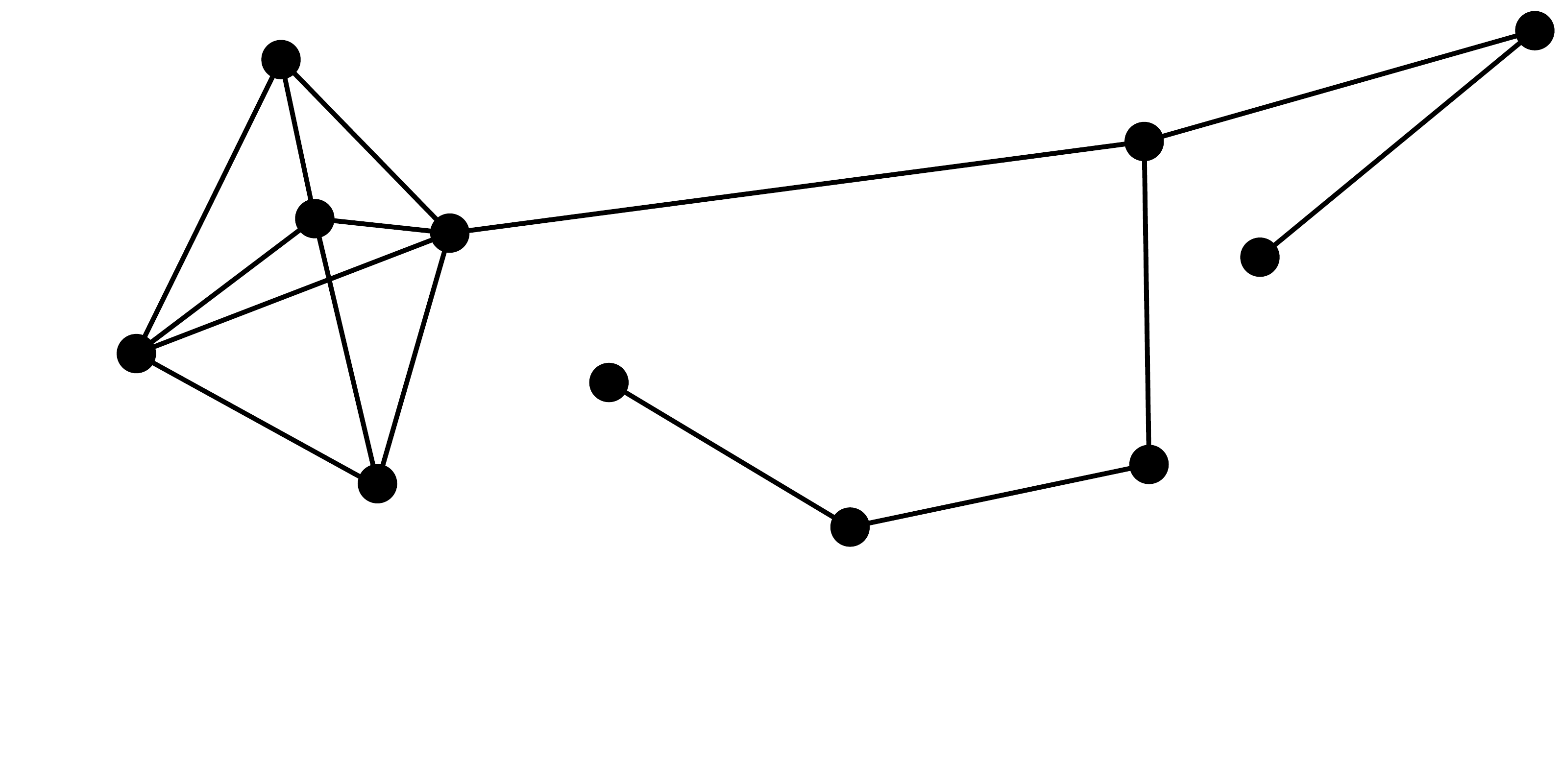}
&
\includegraphics[width=8.5cm]{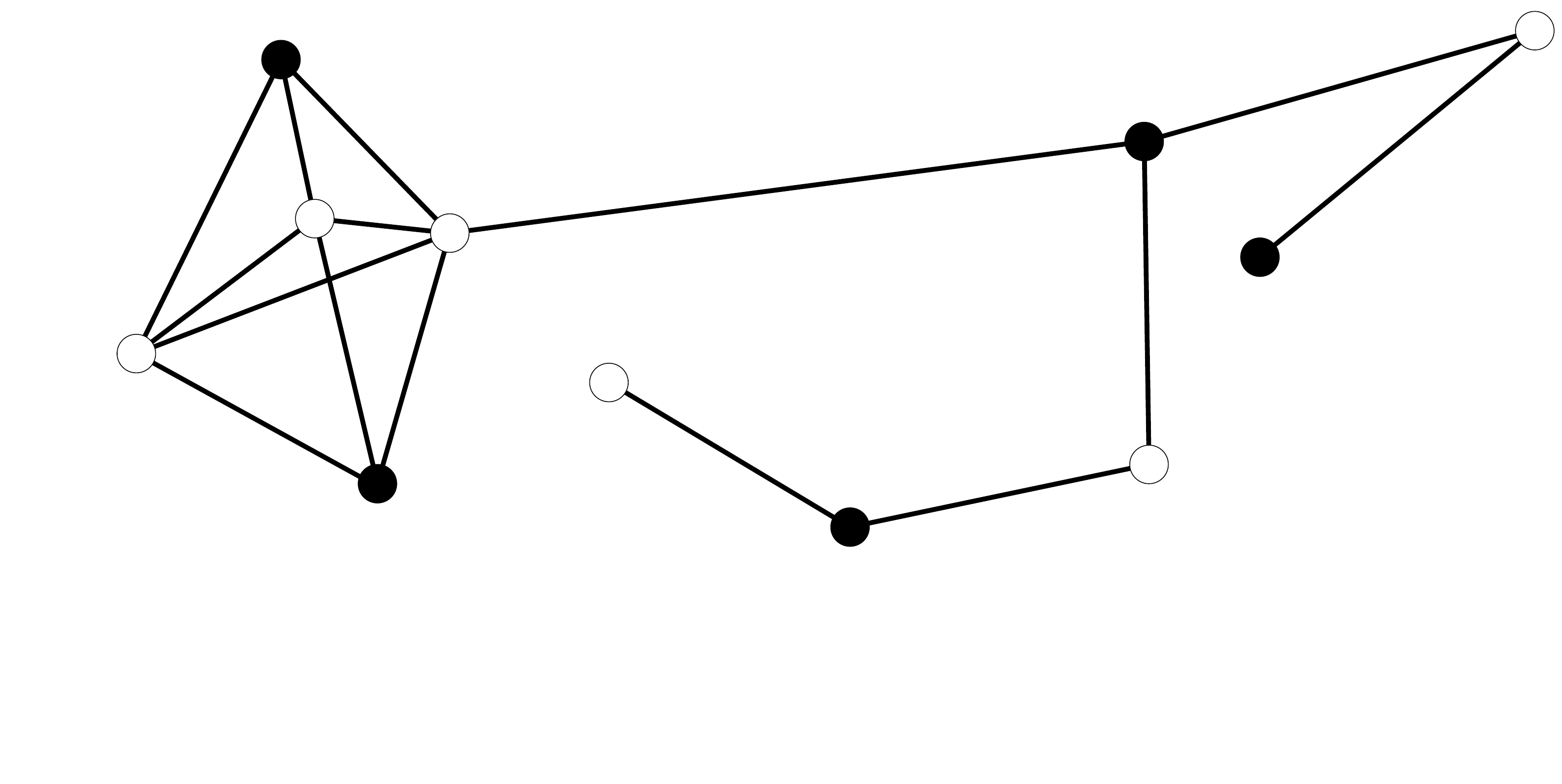}
\end{tabular}
\caption{}
\label{Fig8}
\end{figure}

\item[Step V:] We construct a macrostructure subgraph for each node from $\mathcal{V}^{2}$ by assembling all structure subgraphs 
whose focal nodes are at a graph distance smaller or equal to $1$ in $K^{1}$. 

This step is illustrated in Fig.~\ref{Fig9}--Fig.~\ref{Fig13}. On the top subfigures the macrostructure subgraphs are depicted with solid lines within 
the original graph, on the lower left subfigure for clarity they are displayed by themselves. 
The lower right subfigures show only these nodes from the macrostructure subgraphs that belong to the MIS and two nodes have been connected by an edge 
under the condition that there has been a path between them in the macrostructure subgraph consisting only of nodes that are not in $\mathcal{V}^1$. 

\begin{figure}[h!]
\begin{tabular}{c}
\includegraphics[width=8.5cm]{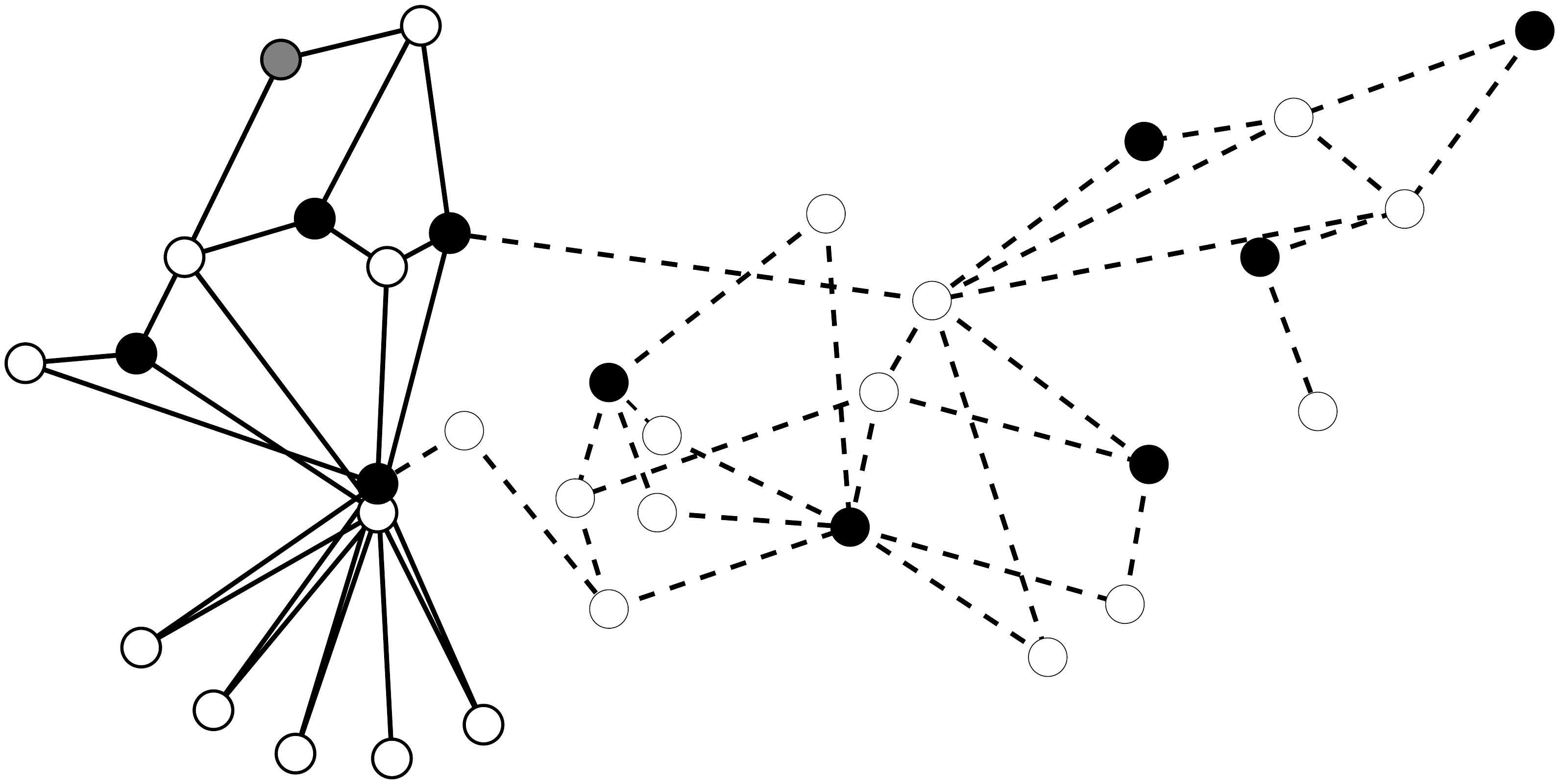}
\\[-0.2ex]
\begin{tabular}{@{}c@{\hspace{.1cm}}c@{}}
\includegraphics[width=8.5cm]{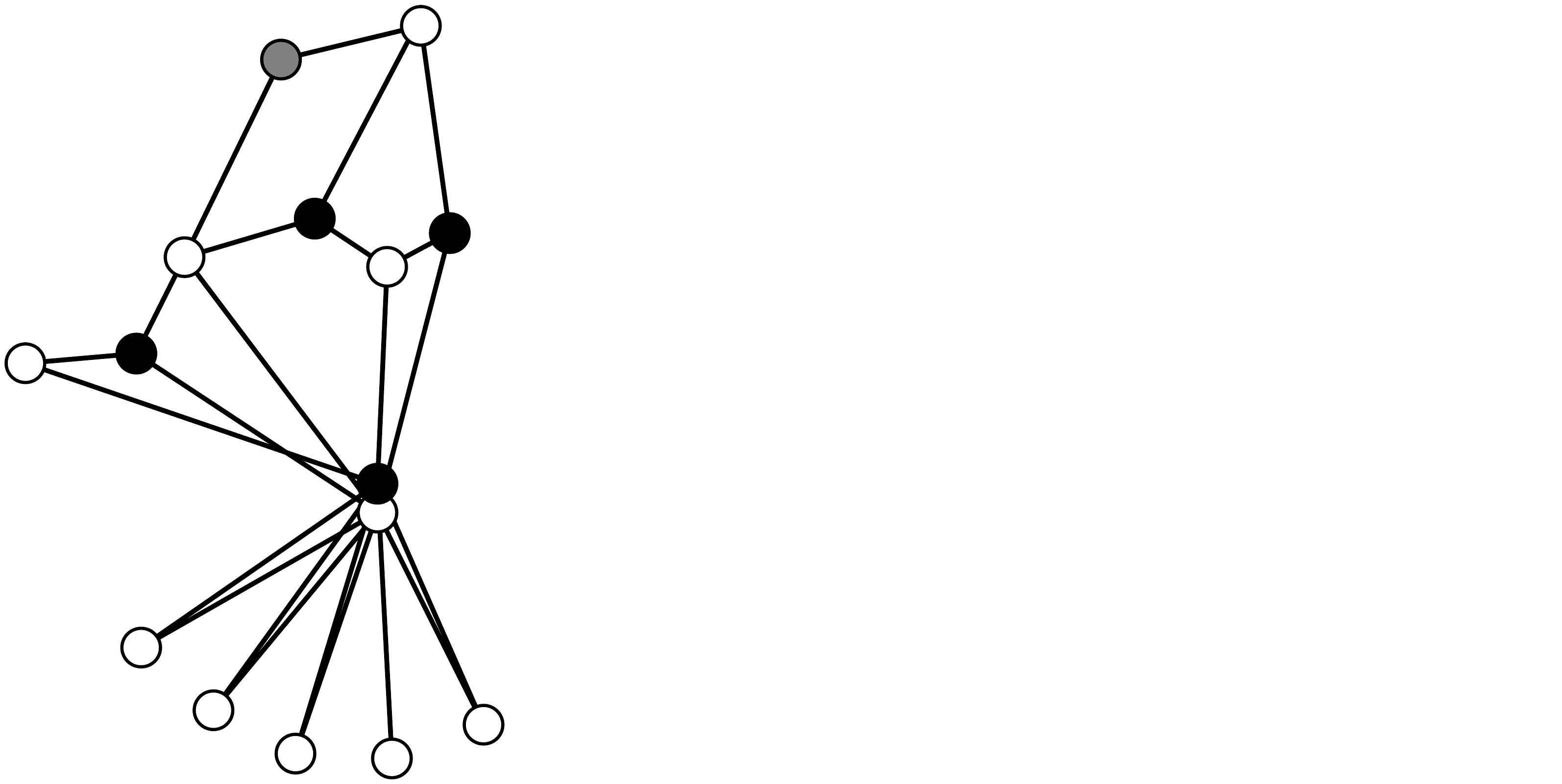}
&
\includegraphics[width=8.5cm]{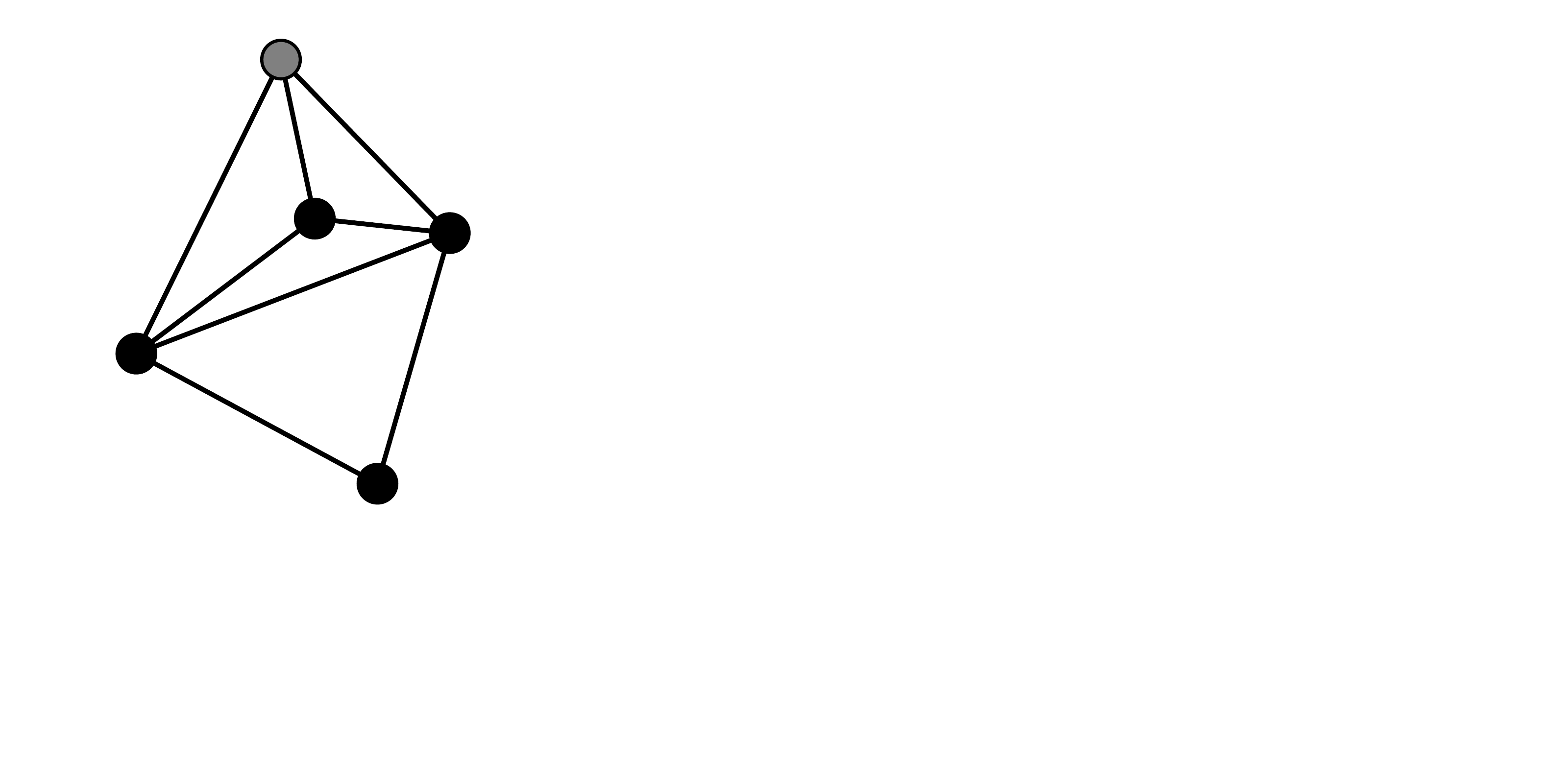}
\end{tabular}
\end{tabular}
\caption{}
\label{Fig9}
\end{figure}

\begin{figure}[h!]
\begin{tabular}{c}
\includegraphics[width=8.5cm]{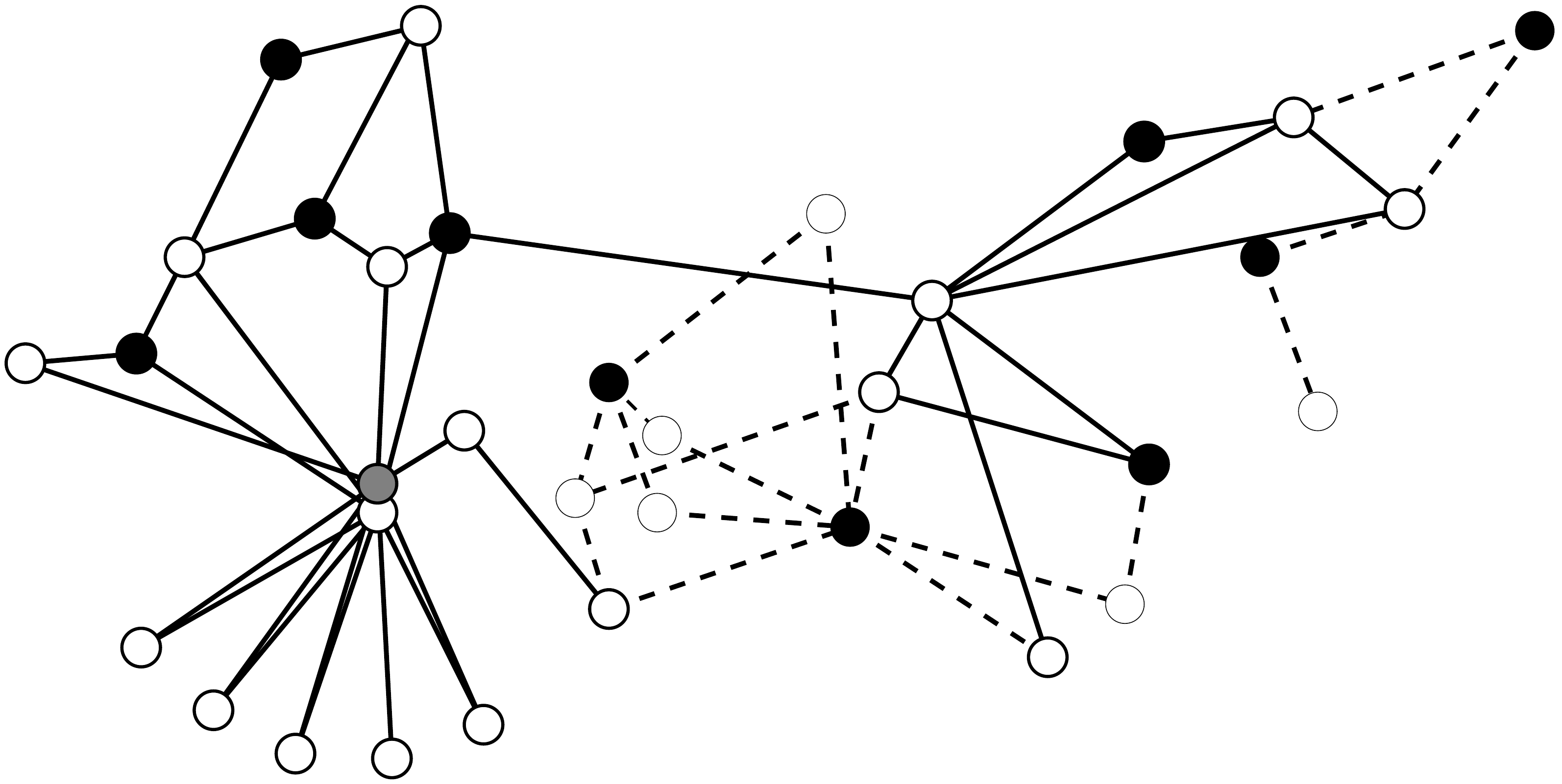}
\\[-0.2ex]
\begin{tabular}{@{}c@{\hspace{.1cm}}c@{}}
\includegraphics[width=8.5cm]{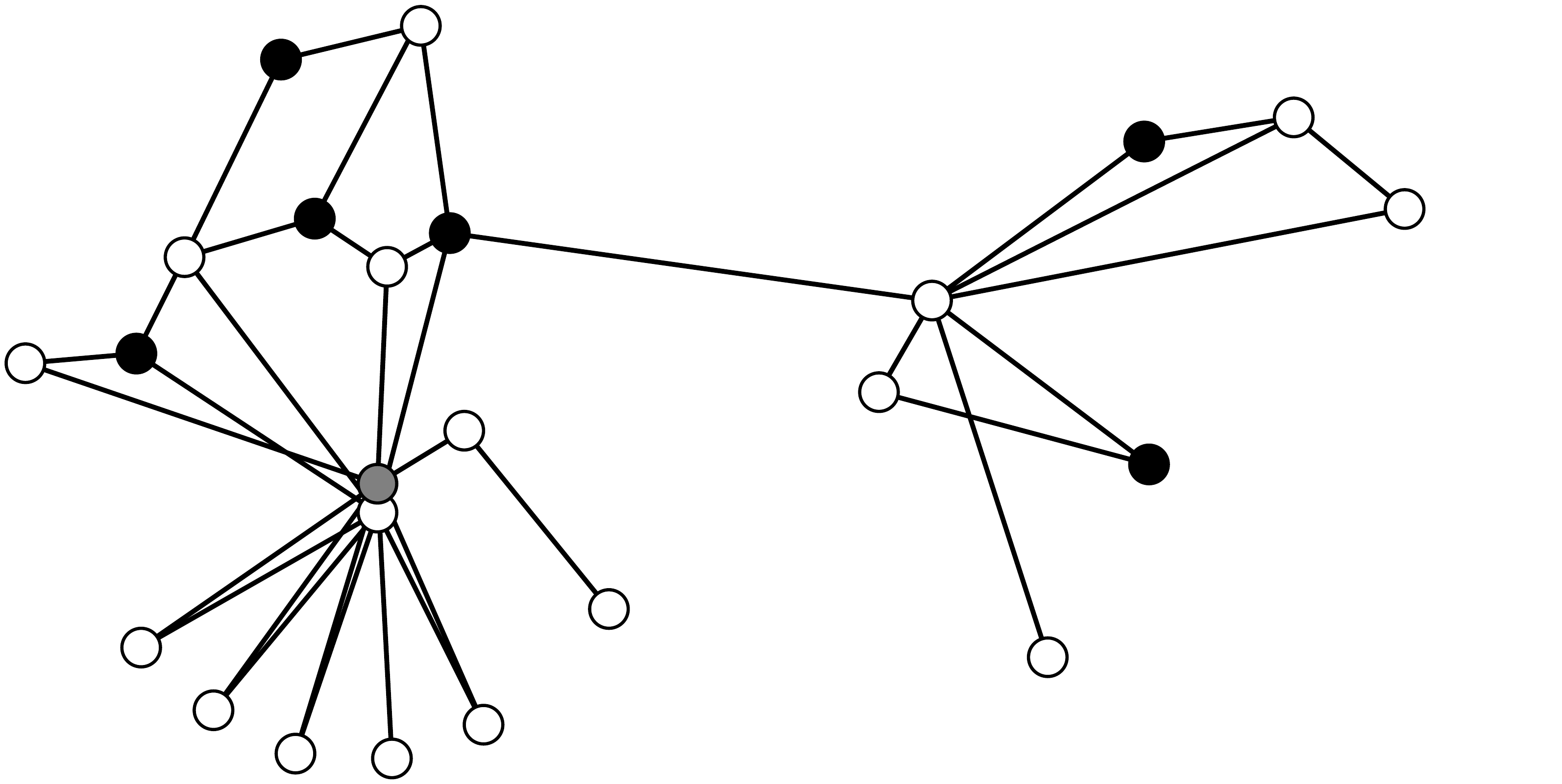}
&
\includegraphics[width=8.5cm]{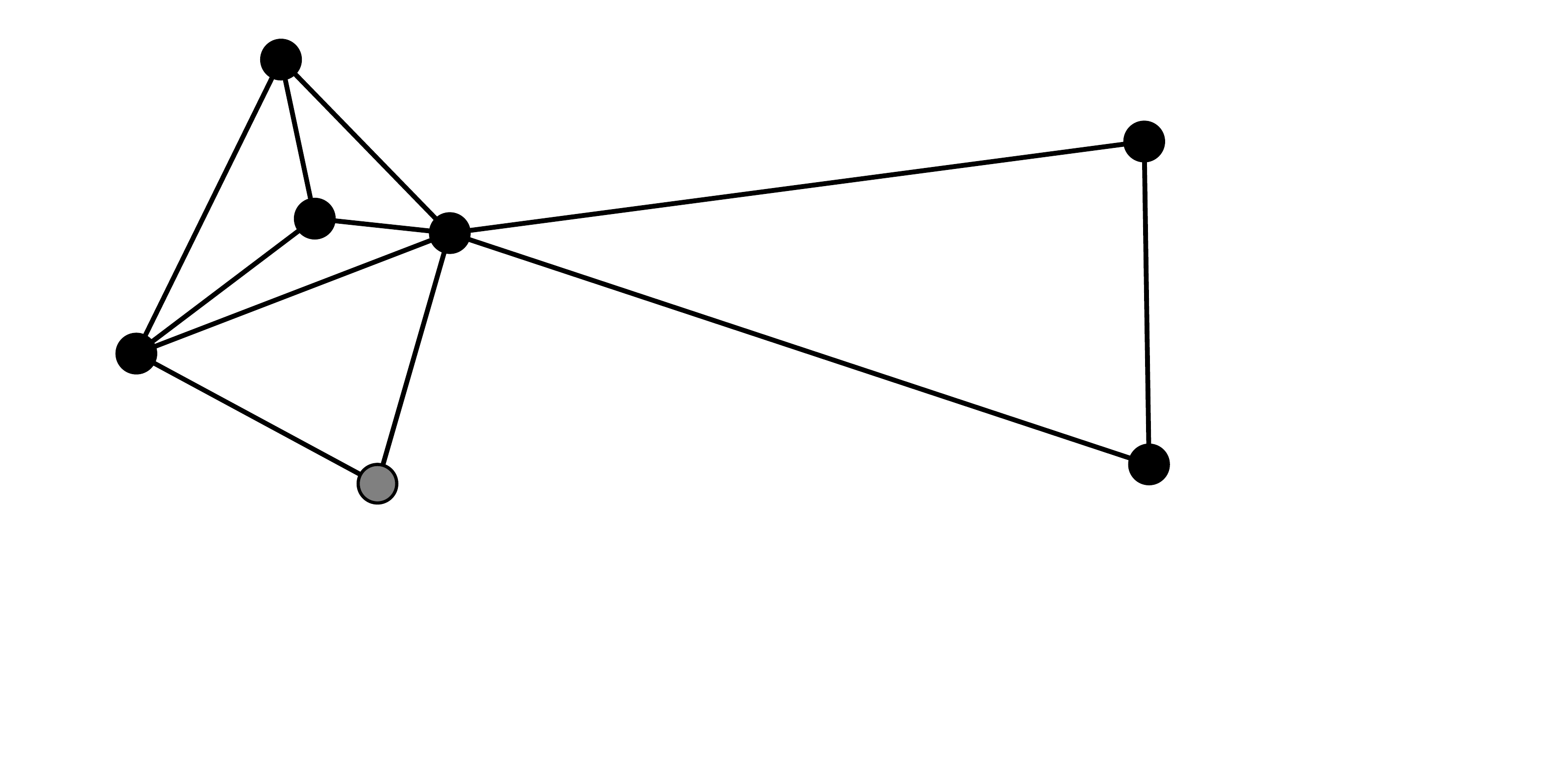}
\end{tabular}
\end{tabular}
\caption{}
\label{Fig10}
\end{figure}

\begin{figure}[h!]
\begin{tabular}{c}
\includegraphics[width=8.5cm]{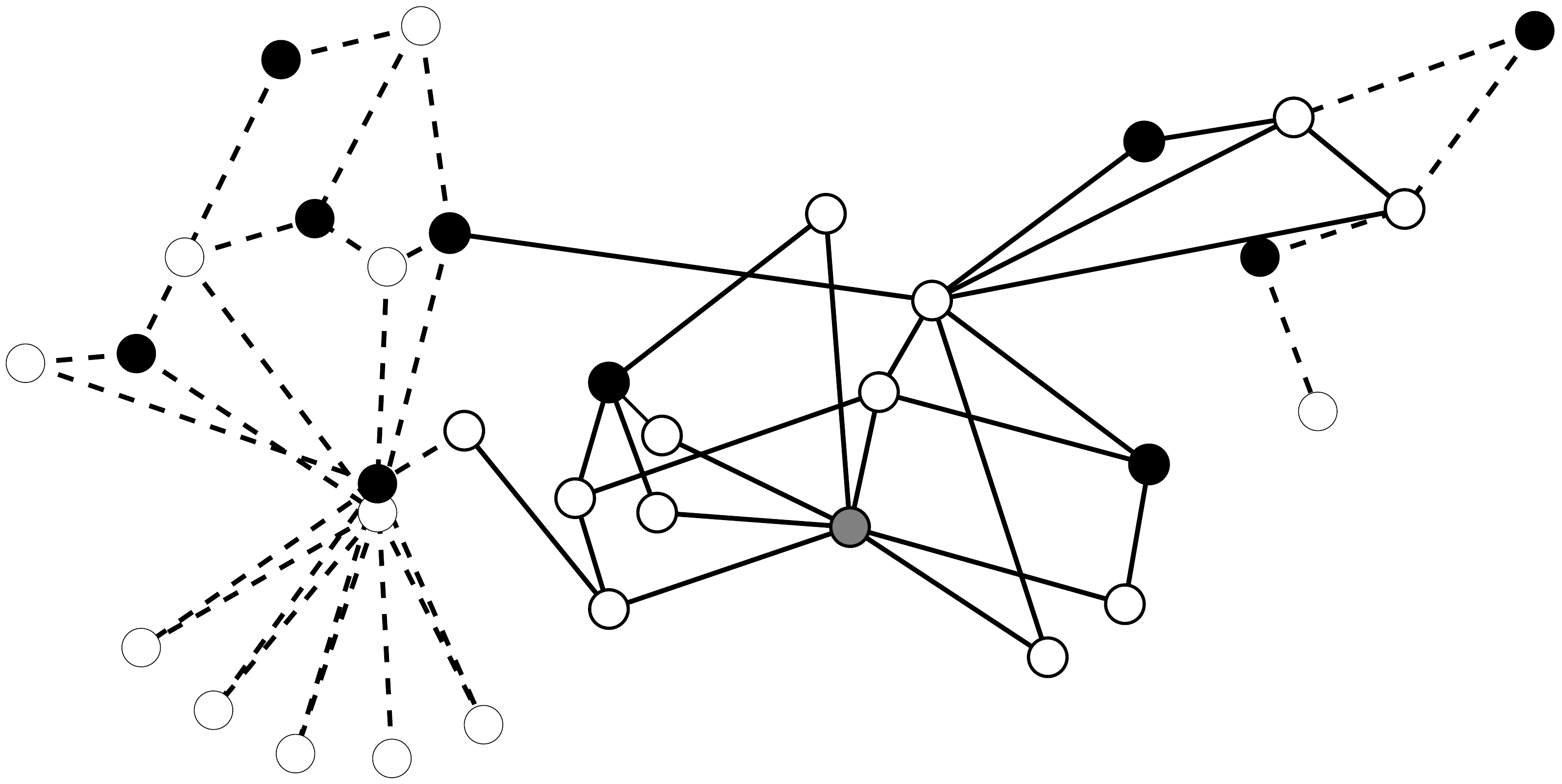}
\\[-0.2ex]
\begin{tabular}{@{}c@{\hspace{.1cm}}c@{}}
\includegraphics[width=8.5cm]{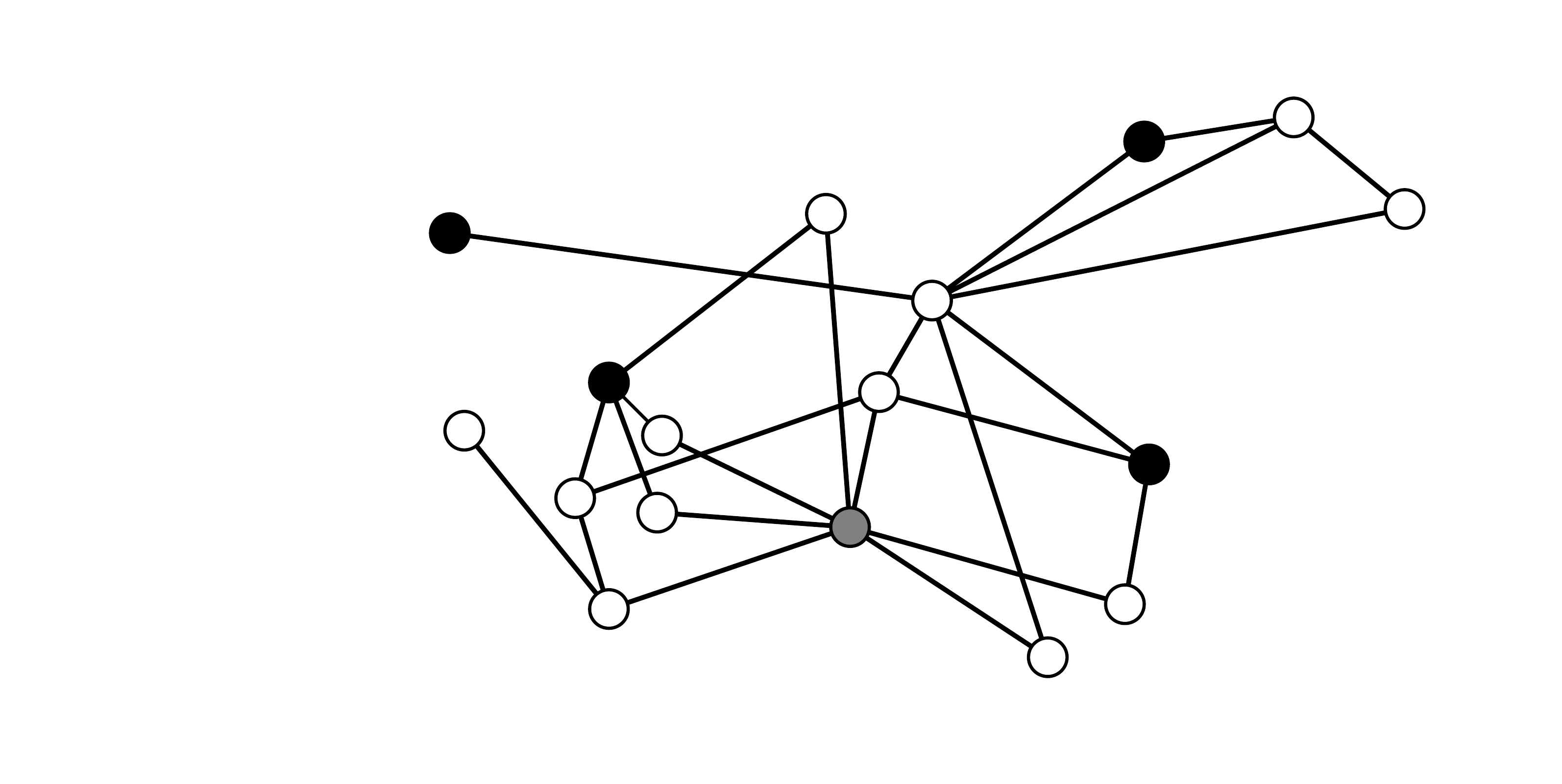}
&
\includegraphics[width=8.5cm]{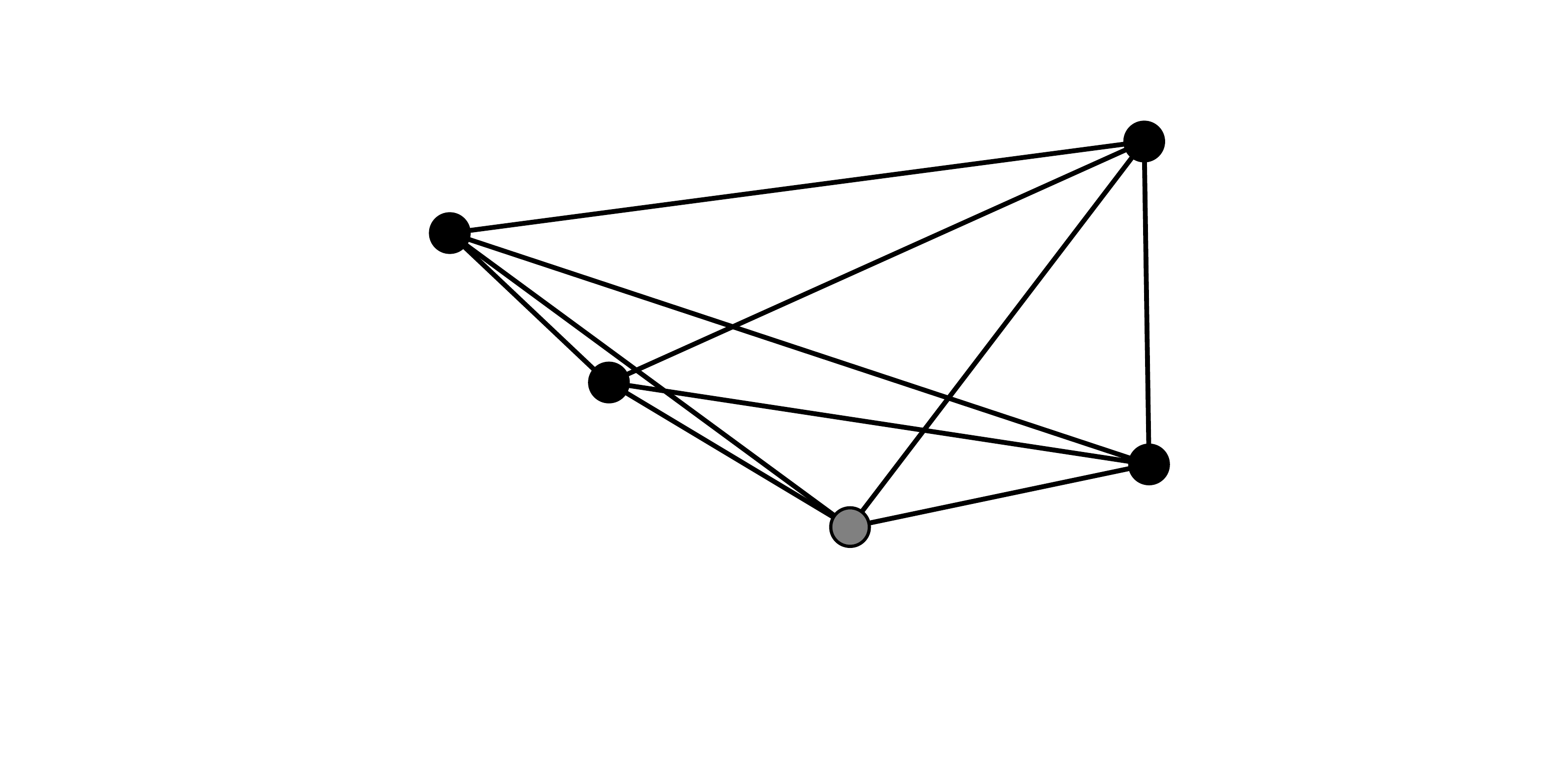}
\end{tabular}
\end{tabular}
\caption{}
\label{Fig11}
\end{figure}

\begin{figure}[h!]
\begin{tabular}{c}
\includegraphics[width=8.5cm]{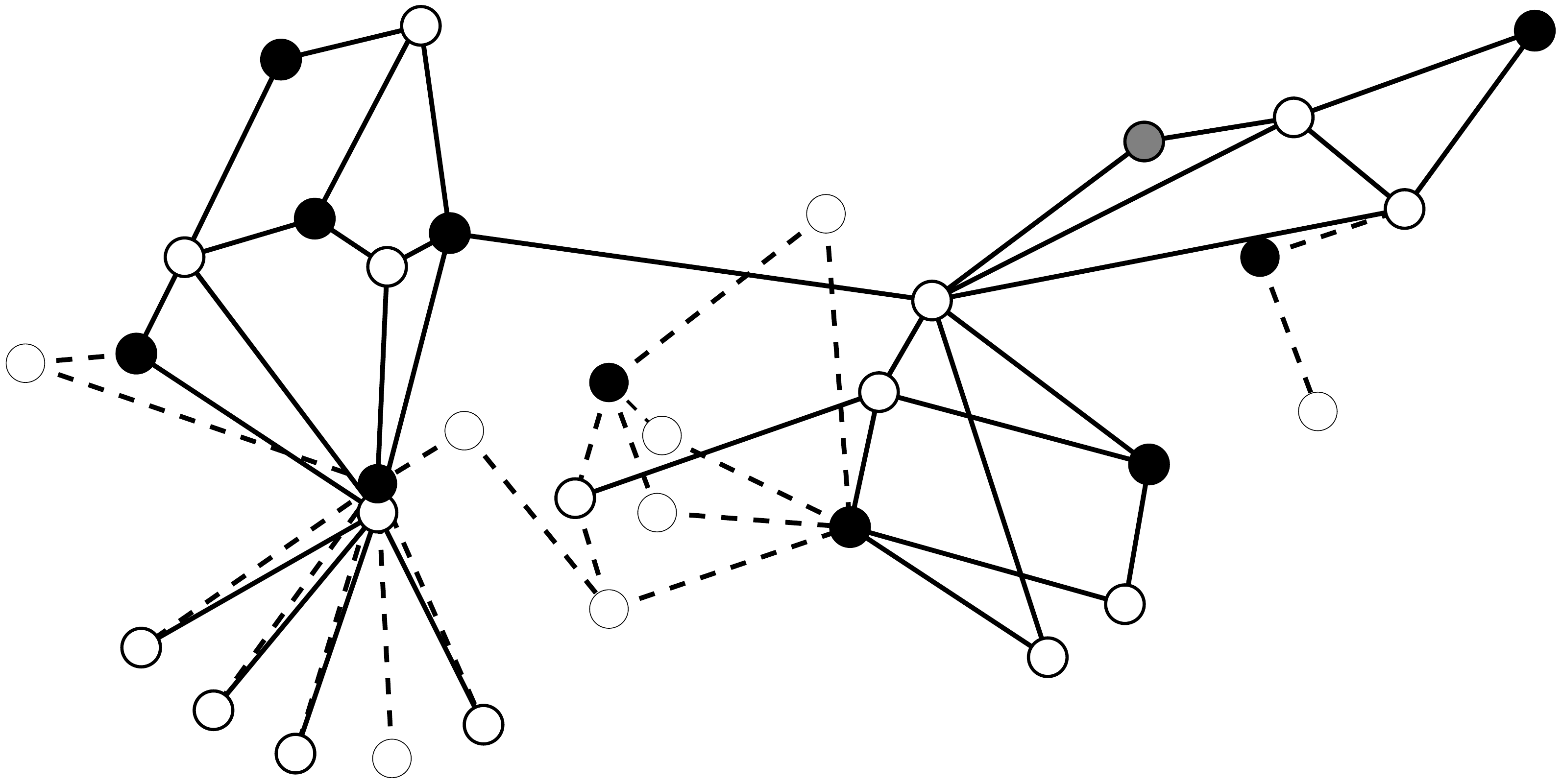}
\\[-0.2ex]
\begin{tabular}{@{}c@{\hspace{.1cm}}c@{}}
\includegraphics[width=8.5cm]{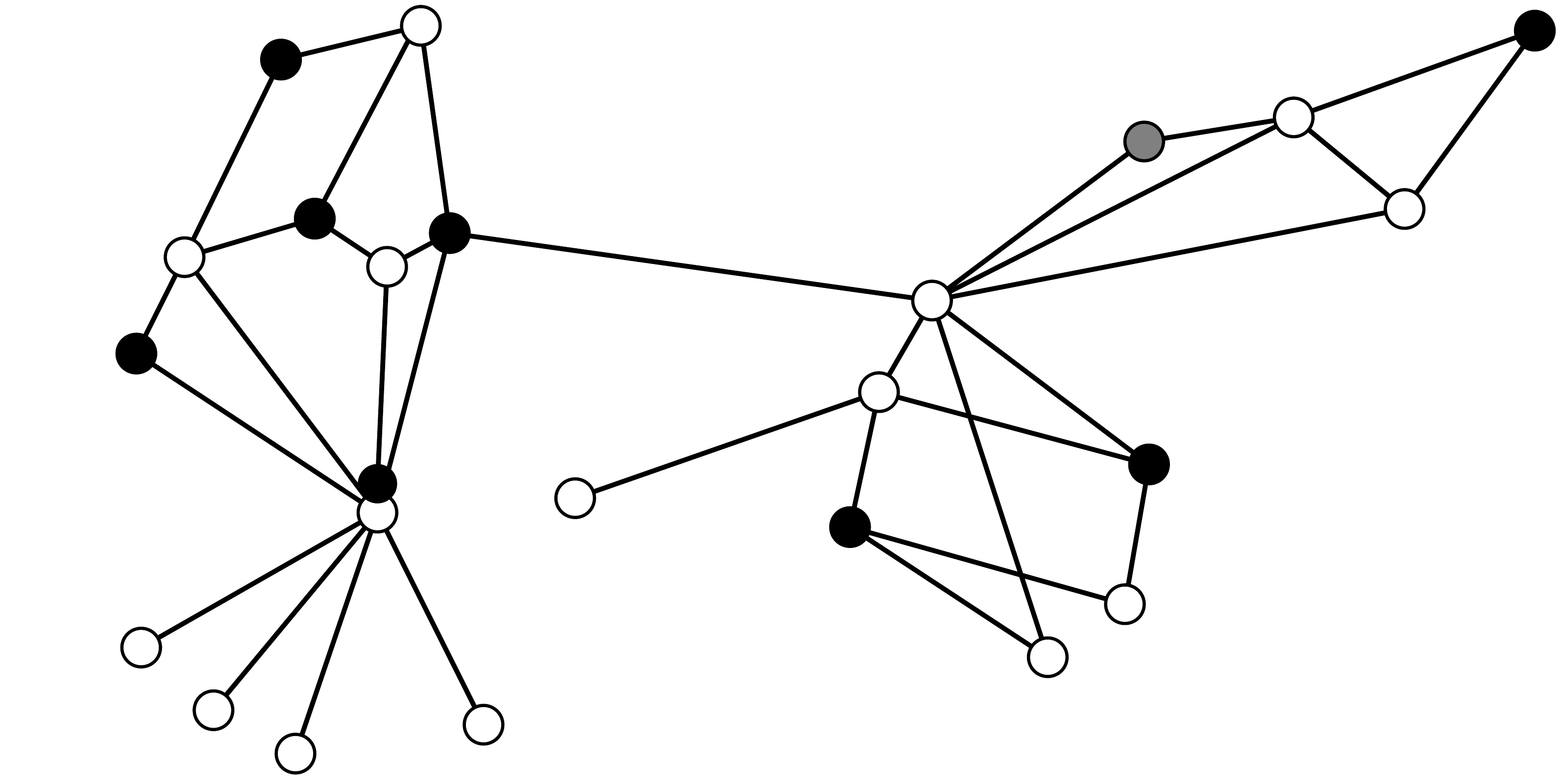}
&
\includegraphics[width=8.5cm]{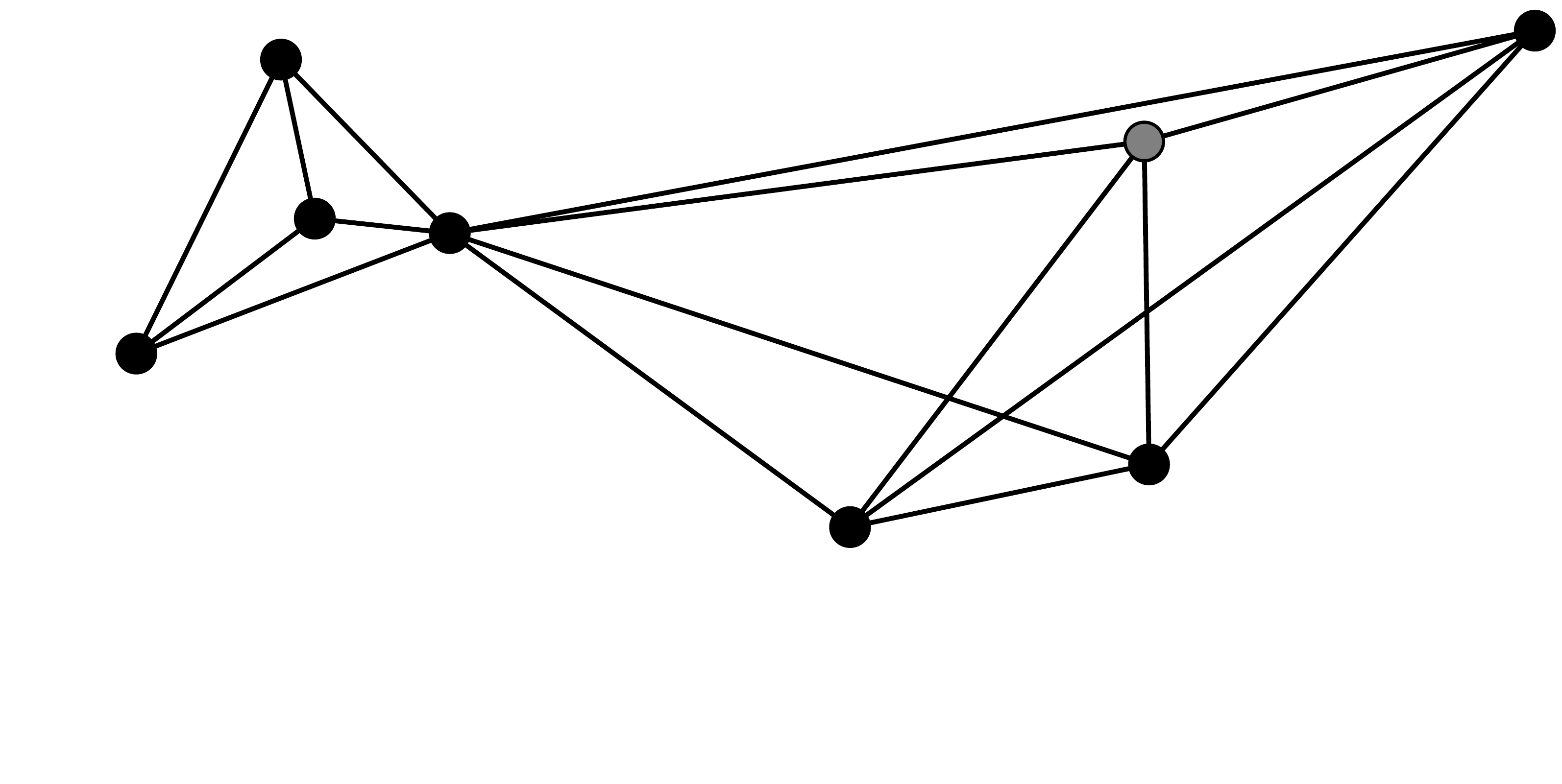}
\end{tabular}
\end{tabular}
\caption{}
\label{Fig12}
\end{figure}

\begin{figure}[h!]
\begin{tabular}{c}
\includegraphics[width=8.5cm]{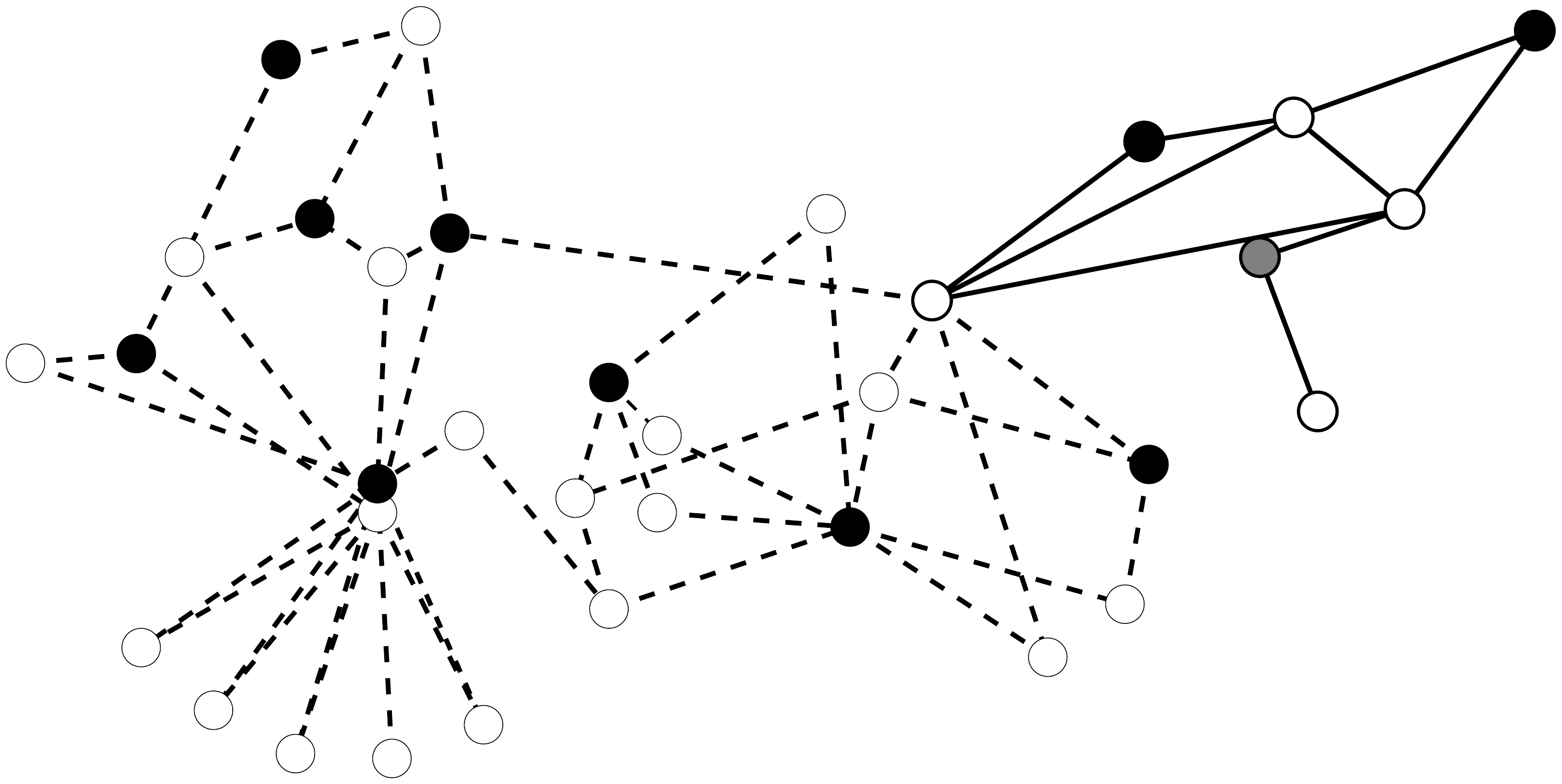} \\
\begin{tabular}{@{}c@{\hspace{.1cm}}c@{}}
\includegraphics[width=8.5cm]{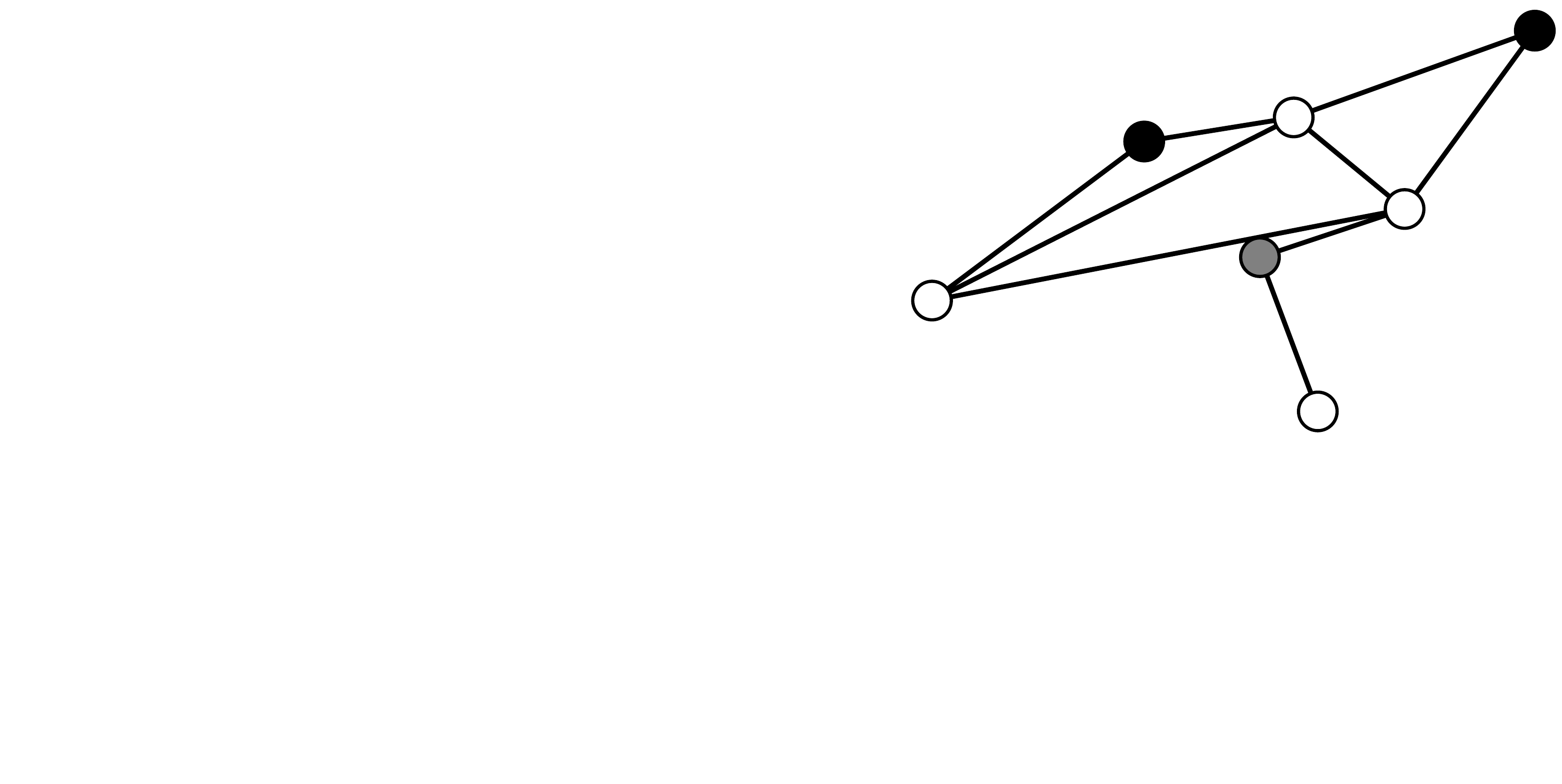}
&
\includegraphics[width=8.5cm]{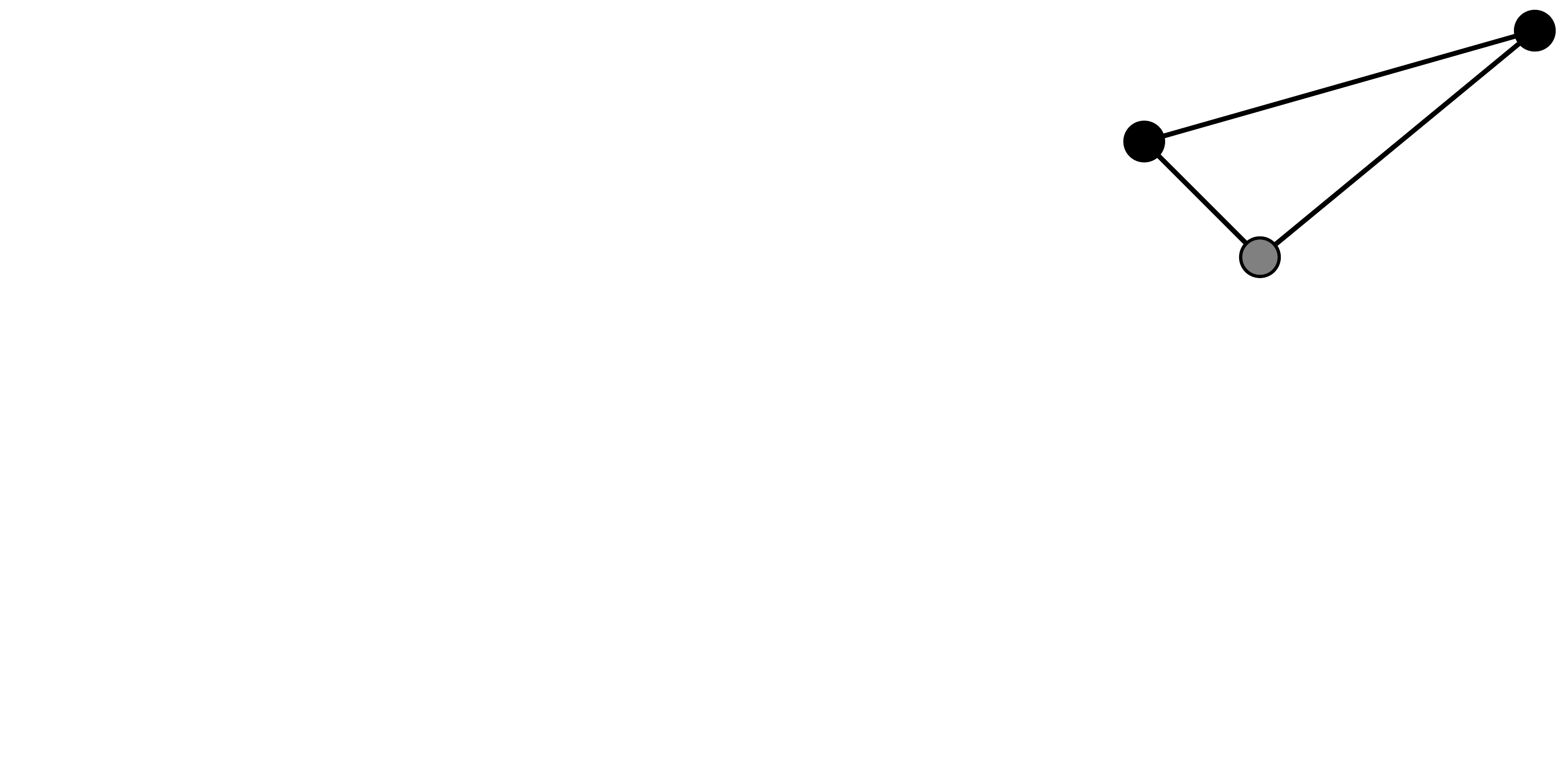}
\end{tabular}
\end{tabular}
\caption{}
\label{Fig13}
\end{figure}
\end{itemize}

There are $3$ remarks to be made regarding the defined procedure:

\begin{remark} 
The very last step in the procedure implicates the construction of two coarse graphs with nodes belonging only to 
$\mathcal{V}^1$. The first one shown on the left of Figure~\ref{Fig14} is the graph assembled from all graphs depicted 
in the lower right subfigures of Fig.~\ref{Fig9}--Fig.~\ref{Fig13}. In the second graph illustrated on the right of Figure~\ref{Fig14} 
two nodes are adjacent if there is a path between them in the original graph $K$ consisting only of nodes that are not in $\mathcal{V}^1$. 
Note that the graph on the right of Fig.~\ref{Fig14} is the adjacency graph of the global Schur complement which results from eliminating 
all nodes that do not belong to $\mathcal{V}^1$.  

As one could observe, the left graph has a sparser structure than the right one and this difference will be much more pronounced on bigger graphs when 
the size of the macrostructures is small as compared to the size of the entire graph.

\begin{figure}
\begin{tabular}{@{}c@{\hspace{.1cm}}c@{}}
\includegraphics[width=8.5cm]{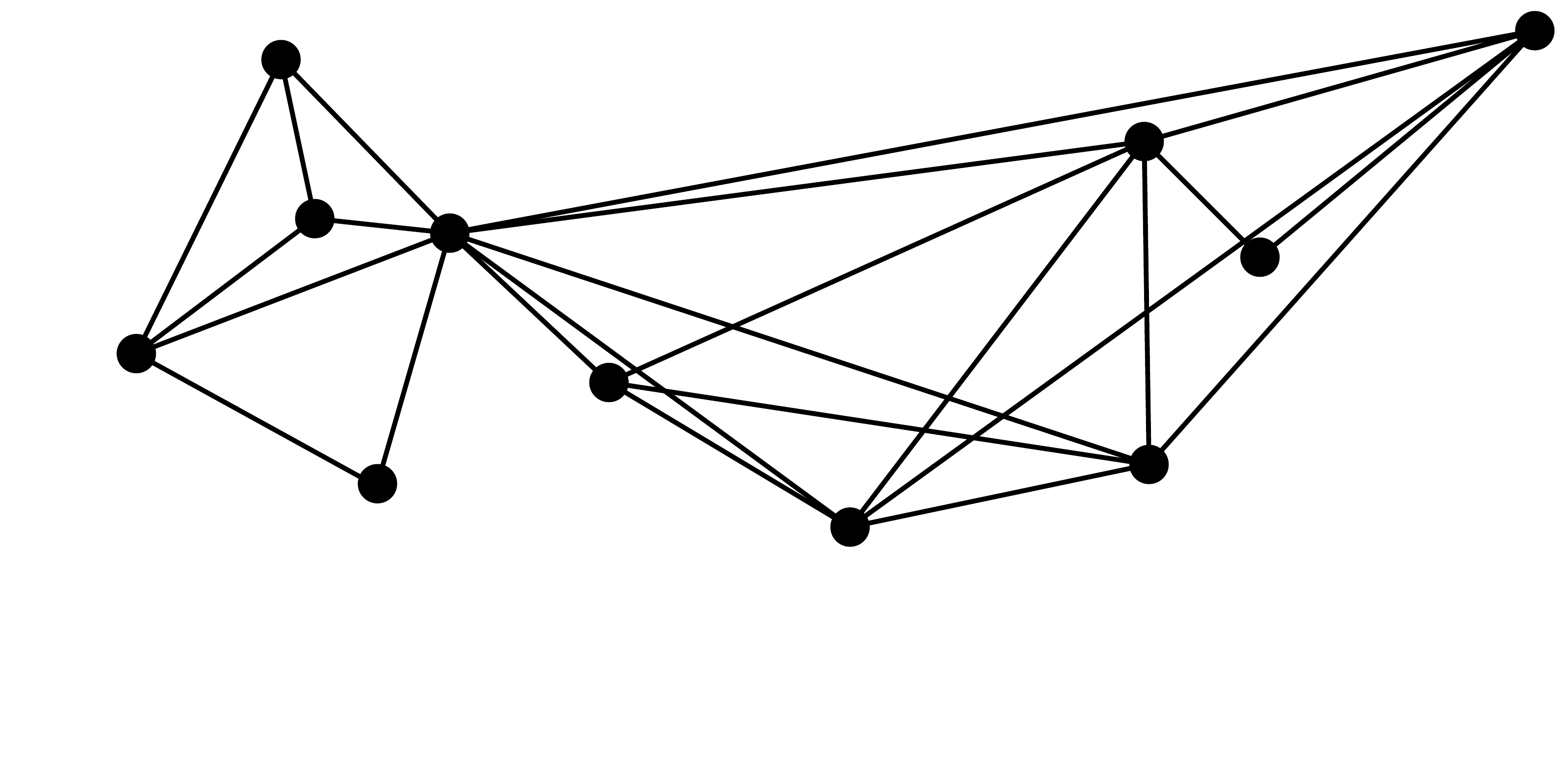}
&
\includegraphics[width=8.5cm]{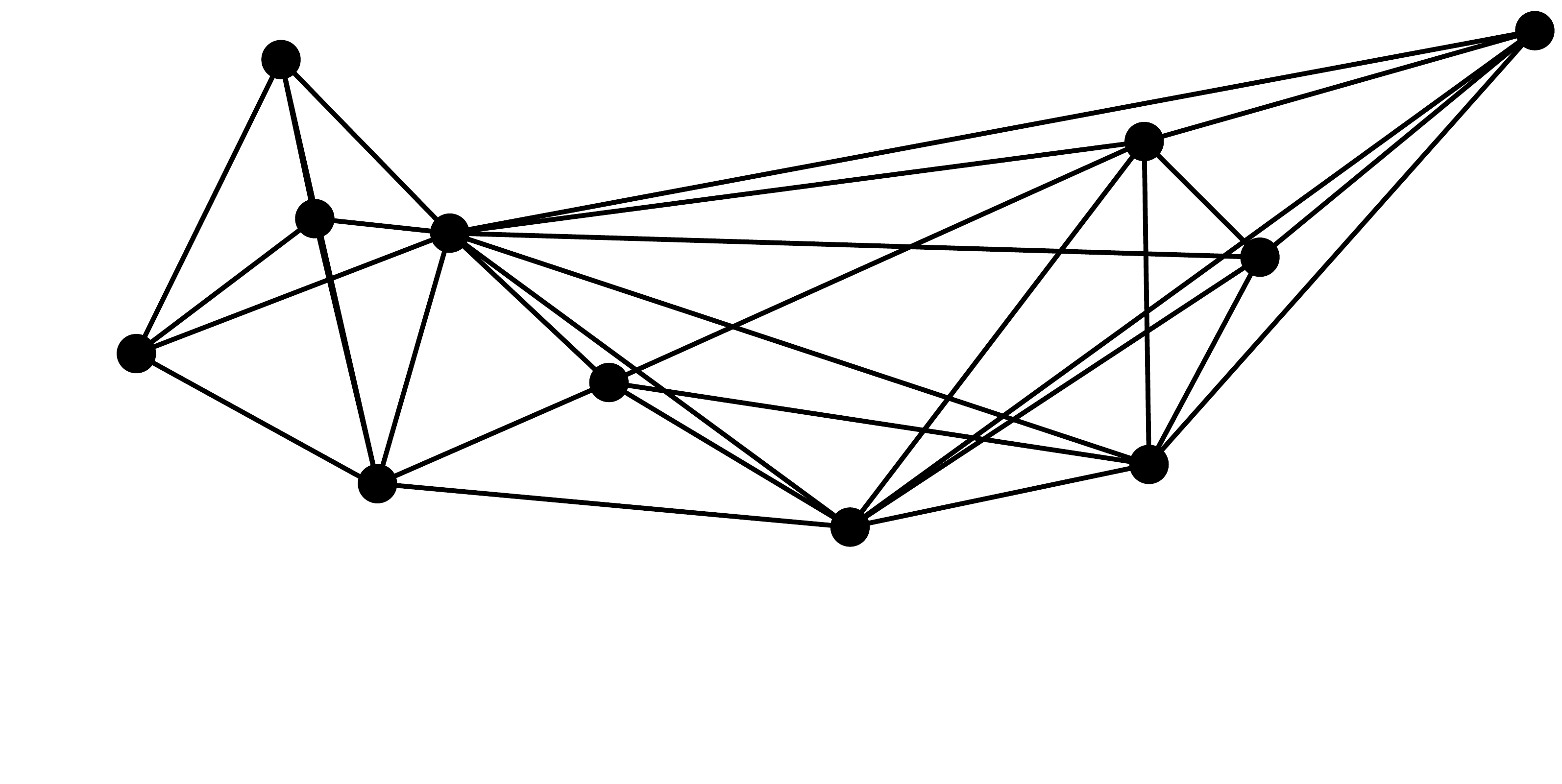}
\end{tabular}
\caption{}
\label{Fig14}
\end{figure}

\end{remark}

\begin{remark}
The described procedure fulfils the requirements defined in \eqref{eq:defs} and the coarser graph $K^1$ is connected. These statements are left as a remark as their proof is evident.
\end{remark}

\begin{remark}
There are also other ways to construct the coverings $\mathcal{G}$ and $\mathcal{F}$ of $K$. For example, 
in Step V one could choose the graph distance according to which the macrostructure subgraphs are assembled to be $2$ 
instead of $1$. It is possible to apply other assembling conditions and restrictions so long as condition~\eqref{eq:defs} is fulfilled. 
\end{remark}

\section{Auxiliary space multigrid (ASMG)}\label{ASMG}

In matrix notation the above described procedure represents the following algorithm:
\vspace{2ex}

\begin{algorithm}\label{alg}
\hrule\vspace{1ex}
\noindent
Additive Schur Complement Approximation (ASCA):
\vspace{1ex}
\hrule
\begin{enumerate}
\item Starting with a global two-level splitting of the degrees of freedom $\mathcal{D}$ into ``coarse'' ($\mathcal{D}_c$) and 
``fine''  ($\mathcal{D}_f$), that is
$$\mathcal{D} = \mathcal{D}_c \oplus \mathcal{D}_f,$$
find a covering $\mathcal{G}$ of $\mathcal{F}$ and a set of scaling
factors $\{ \sigma_{F,G} \}$.
\item For all $K_G \in \mathcal{G}$ execute the following steps:
\begin{itemize}
\item[(a)] Fix the ``local'' two-level numbering of DOF of $K_G$ 
to obtain
$$
A_G = \left[
\begin{array}{cc}
A_{G:11} & A_{G:12} \\[2mm]
A_{G:21} & A_{G:22}
\end{array}
\right]
\begin{array}{l}
\} \  \mathcal{D}_{G:c} \\[2mm]
\} \  \mathcal{D}_{G:f}
\end{array}
$$

\item[(b)] Compute the ``local'' Schur complement
$
S_G = A_{G:11} - A_{G:12} A_{G:22}^{-1} A_{G:21}.
$
\item[(c)] Determine the 
``local-to-global'' mapping $R_{G:1}=R_{G:c}$ for the 
CDOF in $\mathcal{D}_{G:c}$.
\end{itemize}
\item Compose the global Schur complement approximation $Q$ from ``local'' exact Schur complements $S_G$, i.e.,
$$
Q = \sum_{K_G \in \mathcal{G}} R_{G:c}^T S_G R_{G:c} = \sum_{K_G \in \mathcal{G}} R_{G:1}^T S_G R_{G:1}.
$$ 

Here the operators $R_{G:c}$ map a global vector from the space defined by $K^{1}$ to the local spaces 
related to the CDOF in the subgraphs $K_G$.

\end{enumerate}
\hrule
\end{algorithm}

\begin{remark}
In the ASCA algorithm the degrees of freedom $\mathcal{D}$ are vertex degrees of freedom and the two-level splitting 
$\mathcal{D} = \mathcal{D}_f \oplus \mathcal{D}_c$ is defined in Step I from the previous section.
\end{remark}

Let $n_1:=\vert\mathcal{D}_{\rm c} \vert$, $n_2:=\vert\mathcal{D}_{\rm f} \vert$ be the cardinalities of $\mathcal{D}_{\rm c}$ and $\mathcal{D}_{\rm f}$. 
We denote with $n_{G:1}$, $n_{G:2}$ the number of CDOF and FDOF associated with
$K_G$, i.e. $n_G=n_{G:1}+n_{G:2}$ where $\dim(V_G)=n_G$.

We introduce the auxiliary space $\widetilde{V}$ 
of size $\tilde{n}=n_1+(\sum_{i=1}^{n_{\mathcal G}} n_{G_i:2})$ and
a surjective mapping $\Pi : \widetilde{V} \rightarrow V$. The inclusion mapping $R^T: V \rightarrow \widetilde{V}$ has the form:
$$
R^T=\left[ \begin{array}{cc}
         I_1   & 0 \\ 0 & R_2^T
           \end{array}
\right] ,\qquad\text{where}
\qquad
R_2^T=\left[ \begin{array}{c}
           R_{1:2} \\ R_{2:2} \\ \vdots \\ R_{n_{\mathcal G}:2}
           \end{array}
\right]\in\mathbb{R}^{n_1\times\left(\sum_{i=1}^{n_{\mathcal G}} n_{G_i:2}\right)}.
$$
The matrix 
$$
\widetilde{A}:=
\left[\begin{array}{cc}
\widetilde{A}_{11} & \widetilde{A}_{12} \\
\widetilde{A}_{21} & \widetilde{A}_{22}
\end{array}
\right]$$
with blocks
$$
\widetilde{A}_{11}:=\sum_{i=1}^{n_{\mathcal G}} R^T_{i:1} A_{G_i:11} R_{i:1} ,\qquad
\widetilde{A}_{22}:=
\left[ \begin{array}{cccc}
 A_{G_1:22} &&&  \\[0.5ex]
& A_{G_2:22} &&\\[0.5ex]
& & \ddots &  \\[0.5ex]
&&& A_{G_{n_{\mathcal G}}:22}   \\[0.5ex]
                \end{array}
\right],
$$
$$
\widetilde{A}_{12}: =[R_{1:1}^T A_{G_1:12} , R_{2:1}^T A_{G_2:12}, \hdots , R^T_{n_{\mathcal G}:1} A_{G_{n_{\mathcal G}}:12}],$$
$$\widetilde{A}_{21}: =[R_{1:1}^T A_{G_1:12} , R_{2:1}^T A_{G_2:12}, \hdots , R^T_{n_{\mathcal G}:1} A_{G_{n_{\mathcal G}}:12}]^T 
$$
is SPSD and defines an energy inner product on the auxiliary space $\widetilde{V}$. 
It is evident that $A=R \widetilde{A} R^T$ and $\widetilde{A}_{22}$ is a block-diagonal matrix
whose blocks are of size $n_{G_i:2}{\times}n_{G_i:2}$ for $i=1,2,\ldots,n_{\mathcal G}$.

\subsection{Auxiliary space method}
\subsubsection{Two-grid preconditioner}\label{sec:as-2-grid-precond}

The auxiliary space preconditioner $C$ of $A$, see e.g.~\cite{K-15,X-96}, is defined as
\begin{equation}\label{eq:two-grid_0}
C^{-1}=\Pi_{\widetilde{D}} \widetilde{A}^{-1} \Pi^T_{\widetilde{D}},
\end{equation}
where the surjective mapping $\Pi_{\widetilde{D}}: \widetilde{V} \rightarrow V$ is given by
\begin{equation}\label{eq:Pi}
\Pi_{\widetilde{D}}=(R \widetilde{D} R^T)^{-1} R \widetilde{D} \qquad \text{and} \qquad
\widetilde{D}=\left[
             \begin{array}{cc}
              I & 0 \\
              0 &  \widetilde{D}_{22}
             \end{array}
           \right].
\end{equation}

The preconditioner~(\ref{eq:two-grid_0}) can be further generalized, see e.g.~\cite{K-16}, as 
\begin{equation}\label{eq:two-grid_3}
B^{-1} := \overline{M}^{-1} + (I - M^{-T} A) C^{-1} (I - A M^{-1}),
\end{equation}
see e.g.~\cite{X-96}, where $C$ is given by~\eqref{eq:two-grid_0}, $M$ is an $A$-norm convergent smoother and
$\overline{M}=M(M+M^T-A)^{-1}M^T$ is the corresponding symmetrised smoother. 
Results related to the condition number estimate of $B^{-1}A$ can be found in~\cite{K-15}.

\begin{remark}

Different choices for $\widetilde{D}_{22}$ in~\eqref{eq:Pi} are possible, for example $\widetilde{D}_{22}=\widetilde{A}_{22}$,
$\widetilde{D}_{22}={\rm diag}(\widetilde{A}_{22})$ or $\widetilde{D}_{22}={\rm tridiag}(\widetilde{A}_{22})$. Discussion about the effect of these choices 
on the convergence and complexity of the method can be found in~\cite{K-15}.
\end{remark}

\subsubsection{Multigrid precondtioner}\label{sec:asmg}

The procedure in Section~\ref{sec:prel_not} can be recursively applied thus resulting 
in the construction of the nested spaces $\mathcal{V}^{\ell}\subset\ldots\subset \mathcal{V}^0=\mathcal{V}$ related to the graphs 
$K_{\ell}=(\mathcal{V}^{\ell},\mathcal{E}^{\ell}),\ldots ,K_0=K=(\mathcal{V},\mathcal{E})$.  
Then
$$A^{(0)}:=A$$ denotes the Laplacian matrix for the original graph $K$. 
Let the index $k$ be in $\{0,\ldots,\ell-1\}$ and consider 
the sequence of auxiliary space matrices $\widetilde{A}^{(k)}$ in the two-by-two
block factorised form
\begin{equation}\label{factorizationK}
({\widetilde{A}}^{(k)})^{-1} =
(\widetilde{L}^{(k)})^T \widetilde{D}^{(k)} \widetilde{L}^{(k)} , 
\end{equation}
where
$$
\widetilde{L}^{(k)} =
\left [
\begin{array}{cc}
I & -\widetilde{A}^{(k)}_{12} (\widetilde{A}^{(k)}_{22})^{-1} \\
 & I
\end{array}
\right ]  \qquad \text{and}\qquad
\widetilde{D}^{(k)} =
\left [
\begin{array}{cc}
{Q^{(k)}}^{-1} & \\
& (\widetilde{A}^{(k)}_{22})^{-1} 
\end{array}
\right ].
$$
The next coarser level matrix  $A^{(k+1)}$ is defined through the
additive Schur complement approximation $Q^{(k)}$, i.e.
\begin{equation}\label{factorizationK2}
 A^{(k+1)}:=Q^{(k)}.
\end{equation}
The (nonlinear) AMLI-cycle ASMG preconditioner $C^{(k)}$ at level $k$ has the form
\begin{equation}\label{multigrid_preconditioner_0}
{C^{(k)}}^{-1} :=
\Pi^{(k)}
 (\widetilde{L}^{(k)})^T 
 \left [
\begin{array}{cc}
C_{\nu}^{(k+1)} & \\
& \widetilde{A}^{(k)}_{22} 
\end{array}
\right ]^{-1}
\widetilde{L}^{(k)} {\Pi^{(k)}}^T.
\end{equation}
Here $\left[C_{\nu}^{(k+1)}\right]^{-1}$ is an approximation of the inverse of $A^{(k+1)}$ where at the coarsest level we set
\begin{equation}
\left[C_{\nu}^{(\ell)}\right]^{-1} := {A^{(\ell)}}^{-1}
\end{equation}
while for $ k< \ell-1$ a matrix polynomial of the form
\begin{equation}\label{eq:poly1}
\left[ C_{\nu}^{(k+1)}\right]^{-1} := (I-p^{(k)}({C^{(k+1)}}^{-1}A^{(k+1)})){A^{(k+1)}}^{-1} 
\end{equation}
is used to determine $\left[C_{\nu}^{(k+1)}\right]^{-1}$.

The incorporation of pre- and post-smoothing results in the following AMLI-cycle ASMG preconditioner: 
%
\begin{equation}\label{multigrid_preconditioner}
{B^{(k)}}^{-1} :=
{\overline{M}^{(k)}}^{-1} + (I - {M^{(k)}}^{-T} A^{(k)})
\Pi^{(k)}
 (\widetilde{L}^{(k)})^T  {\overline{D}^{(k)}}^{-1}
\widetilde{L}^{(k)} {\Pi^{(k)}}^T (I - A^{(k)} {M^{(k)}}^{-1}),
 \end{equation}
where
$$
\quad \overline{D}^{(k)} :=
\left [
\begin{array}{cc}
B_{\nu}^{(k+1)}& \\
& \widetilde{A}^{(k)}_{22}
\end{array}
\right ] \quad
\text{and} \quad
\left[ B_{\nu}^{(k+1)} \right]^{-1} = q^{(k)}({B^{(k+1)}}^{-1}A^{(k+1)})){B^{(k+1)}}^{-1}.
$$

We write $[B_{\nu}^{(k+1)}]^{-1}$ as $B_{\nu}^{(k+1)}[\cdot]$  
(and $[C_{\nu}^{(k+1)}]^{-1}$ as $C_{\nu}^{(k+1)}[\cdot])$) to denote that for the nonlinear AMLI-cycle ASMG method
the coarse-level preconditioner is a nonlinear mapping whose action on a vector $\d$ is realised by $\nu$ iterations
with a preconditioned Krylov subspace method.
In what follows, we employ the generalised conjugate gradient method 
and hence denote $B_{\nu}^{(k+1)}[\cdot] \equiv B^{(k+1)}_{\text{GCG}}[\cdot]$
(and $C_{\nu}^{(k+1)}[\cdot] \equiv C^{(k+1)}_{\text{GCG}}[\cdot]$).

\begin{remark}
The method of auxiliary space preconditioning dates back to the works of Matsokin and Nepomnyashchikh, see~
\cite{M-85,N-91,N-95}.
\end{remark}

When applied recursively, the procedure in Section~2.2 generates sequences of macrostructure and structure subgraphs fitted to the multilevel 
splitting of the DOF. It is clear that the focus node of a macrostructure subgraph will always be 
a focus node of a structure subgraph on the next coarser level. However, not all possible ways of defining the structures, macrostructures and the adjacency 
pattern of the coarser graph in Step II, Step IV and Step III respectively will be consistent in the sense that the coarse nodes in a given macrostructure coincide with the 
nodes of the corresponding structure on the next coarser level. 

One simple strategy to avoid such inconsistency is to run the described procedure recursively only to generate the multilevel splitting of the DOF and  
the macrostructure-structure relation needed in~\eqref{s_ms}. To generate the structures on the finest level one could apply a rule, as in Step II of the procedure, 
and then use the obtained structures in the assembly of the finest level macrostructures. 
On all coarser levels the structures could then directly be associated with adjacency graphs of the local Schur complement matrices computed in Step 2(b) of Algorithm~3.1, 
cf. also with~\eqref{factorizationK2}. 

It is worth mentioning that this problem of inconsistency does not appear for the particular choice of defining the structures, macrostructures and 
coarse subgraphs presented in the description of the procedure in Section~2.2. 
As it is easy to be seen, in this particular setting the coarse vertices that a macrostructure subgraph contains are either at a graph distance $2$ or $4$ to 
the focus node in the original graph $K$ and in the coarser graph $K^1$ exactly these and no other coarse nodes are at a graph distance $1$ or $2$ to this node.

\section{Complexity of the ASMG preconditioner}\label{sec:complexity}

The computational complexity of the ASMG preconditioner relates directly to the 
work spent on constructing the precondtioner and work for its utilization. 
The work is determined by the number of arithmetic operations which is mainly 
affected by the sparsity of the involved matrices. For that reason in what follows we want to comment on the 
sparsity of the coarse-level matrices generated by Algorithm~3.1.  

Let $\mathcal{S}$ be the set containing the undirected unweighted graphs $K_{S_G}=(\mathcal{V}_{S_G},\mathcal{E}_{S_G})$ 
corresponding to the adjacency matrices  
of the local Schur complements $S_G$ as computed in Step 2~(b) in the algorithm. Here 
$\mathcal{V}_{S_G}$ and $\mathcal{E}_{S_G}$ designate respectively the set of vertices and the set of edges. 
Obviously, 
$\mathcal{V}_{S_G}\subset \mathcal{V}_G$,
however, $\mathcal{E}_{S_G}$ is not a subset of $\mathcal{E}_G$. Frequently but not always, $K_{S_G}$ will be a so-called clique, i.e., a graph in
which every pair of distinct vertices (DOF) is connected by an edge.

Now, let the coarse-level matrix $A^{(k+1)}$ be defined via \eqref{factorizationK2} where $Q^{(k)}=Q$ results from ASCA, i.e.,  from Algorithm~\ref{alg}, Step 3. 
The following lemma characterizes the sparsity pattern of the coarse-level matrix $Q$. 

\begin{lemma}\label{lem:sparsity}
Let 
$$Q=(q_{ij})_{i,j=1}^{\vert \mathcal{D}_c\vert}= \sum_{K_G \in \mathcal{G}} R_{G:1}^T S_G R_{G:1}.
$$  
The entry $q_{ij}$ of $Q$ is zero for any pair of coarse DOF $d_i$, $d_j\in \mathcal{D}_c$ if there does not exist a graph 
$K_{S_G}\in \mathcal{S}$ for which both $d_i$ and $d_j$ are in $\mathcal{V}_{S_G}\;(\equiv\mathcal{D}_{G:c})$.
\end{lemma}

\begin{proof}

Assume that 
$$q_{ij}=\sum_{K_G\in\mathcal{G}}\langle S_G R_{G:1}\, \e_i, R_{G:1}\, \e_j\rangle\neq 0,$$ 
where  $\e_i$ and $\e_j$ denote the $i$-th and $j$-th
canonical basis vectors in $ \mathbb{R}^{\vert \mathcal{D}_c\vert}$, respectively.
Then there exists $K_G$ for which
$\langle S_G R_{G:1}\, \e_i, R_{G:1}\, \e_j \rangle =\e_j^T R_{G:1}^T S_G R_{G:1} \, \e_i \neq 0$. 

Hence, from the identity $\e_j^T R_{G:1}^T S_G R_{G:1}\, \e_i=\lambda(S_G R_{G:1}\, \e_i \e_j^T R_{G:1}^T)$, where $\lambda(M)$ denotes the 
largest-in-magnitude eigenvalue of any square real matrix $M$, it follows that the rank one matrix $R_{G:1} \e_i \e_j^T R_{G:1}^T$ has to be different 
from the matrix of all zeros. In view of the definition of $R_{G:1}$ the condition $d_i$ and $d_j$ not both belonging to $\mathcal{D}_{G:1}$, however, 
implies that $R_{G:1}\,\e_i \e_j^T R_{G:1}^T=0$ which is a contradiction. 
\end{proof}

The next lemma gives conditions which control the sparsity of the Schur complement approximation $Q$.
\begin{lemma}
Let $(\widetilde{\mathcal{V}},\widetilde{\mathcal{E}})$ be the the (factor) graph representing the adjacency of the graphs belonging to $\mathcal{S}$, defined by 
identifying each graph $K_{S_G}=(\mathcal{V}_{S_G},\mathcal{E}_{S_G})$ with a vertex $v_G\in\widetilde{\mathcal{V}}$ and 
connecting two vertices $v_{G^{\prime}}$ and $v_{G^{\prime\prime}}$ by an edge $\e_{{G^{\prime}}{G^{\prime\prime}}}\in\widetilde{\mathcal{E}}$ if and
only if $\mathcal{V}_{S_{G^{\prime}}}\cap \mathcal{V}_{S_{G^{\prime\prime}}}\neq \emptyset$.

Under the assumptions that
\begin{itemize}
\item[(i)]
$(\widetilde{\mathcal{V}},\widetilde{\mathcal{E}})$ has a bounded vertex degree, i.e., $degree(v_G)\le c_1$ for all 
$v_G\in \widetilde{\mathcal{V}}$,
\end{itemize}
and
\begin{itemize}
\item[(ii)]the size of the graphs $K_{S_G}$ is uniformly bounded, i.e., $\vert \mathcal{V}_{S_G}\vert\le c_2 $ for all $K_G\in\mathcal{G}$, 
\end{itemize}
the coarse-level matrix $Q$ is sparse, i.e., the number of non-zero entries in every row of $Q$ is bounded by a constant $c_3$. 
\end{lemma}

\begin{proof}
Due to property (i) for a given row $k$ the degree of freedom $d_k\in\mathcal{D}_c$ can belong to at most $(c_1+1)$ graphs $K_{S_G}$. Because of Lemma~\ref{lem:sparsity} a 
non-zero entry $q_{km}\neq 0$ can only occur if there exists $K_{S_G}$ such that both $d_k$ and $d_m$ belong to $\mathcal{V}_{S_G}$. However, since there are at most $(c_1+1)$ 
such possibilities and the size of the corresponding graphs is bounded by $c_2$, see assumption (ii), the total number of non-zero entries in row $k$ is bounded by $c_3:=(c_1+1)(c_2-1)+1$.     
\end{proof}

Discussion about the complexity of the ASMG algorithm based on ASCA can also be found in~\cite{K-16}.

\section{Numerical tests}\label{ne}

The convergence performance of the AMLI-cycle ASMG method based on ASCA is shown 
in $8$ examples from 
\begin{verbatim}
http://www.cise.ufl.edu/research/sparse/matrices/list_by_id.html,
\end{verbatim} 
see \cite{D-11}, also \cite{BA-13,D-89} and Figures~\ref{Fig15}--\ref{Fig22}. In Tables~\ref{tab:1}--\ref{tab:8} additionally to the number of iterations 
we report the number of CDOF for all levels. 

The numerical experiments have been run for:
\begin{itemize}
\item[(i)] V- and W-cycles; 
\item[(ii)] $2$ Gauss-Seidel smoothing steps;
\item[(iii)] a random start vector;
\item[(vi)] zero right hand side;
\item[(v)] stopping criterium: relative residual $10^{-8}$;
\item[(vi)] final level of coarsening: number of CDOF  $\le40$.
\end{itemize}

For Examples~1--4, 6--8 the components of the ASMG method based on ASCA are constructed in accordance with the procedure in Section~\ref{sec:prel_not} whereas in  
Example~5, where the graph is very sparse, we have chosen in Step~II of the procedure the set of nodes that defines a 
structure subgraph to consist of a focus node and all nodes that are at a graph distance $1$, $2$, $3$ or $4$ to it.

\subsection*{Example 1}


In this example the graph consists of $7\,928$ nodes and $59\,379$ edges. The minimum vertex degree is $1$, the maximum vertex degree is $33$, 
the average vertex degree is $\approx 14,98$.

\vspace{2ex}
\begin{tabular}{l|l}
Description & skirt, with coordinates. From NASA, collected by Alex Pothen \\
Author & NASA \\
 Editor & G. Kumfert A. Pothen
\end{tabular}
\vspace{2ex}

{
\begin{table}[h!]
 \begin{center}
 \renewcommand{\arraystretch}{1.4}
 \begin{tabular}{c |  c  c  c c   }
\hline
Levels &  2 & 3 & 4 & 5  \\  [1ex]
\hline
V-cycle & 17 & 17 & 18 & 18  \\  [1ex]
W-cycle & 17 & 17 & 17 & 17  \\ [1ex]
 \hline
CDOF & ~2422~ & ~~608~~ & ~~137~~ & ~~~40~~~  \\
\hline
\end{tabular}
\end{center}
\caption{Example 1, $7\,928$ DOF}\label{tab:1}
\end{table}
}

\subsection*{Example 2}


Here, there are 
$10\, 429$ nodes, $46\, 585$ edges and the minimum vertex degree is $3$, the maximum vertex degree is $27$ and 
the average vertex degree is $\approx 8,93$.

\vspace{2ex}
\begin{tabular}{l|l}
Description & SHUTTLE\_EDDY: Nasa matrix, but with diagonal added to original matrix \\
Author & NASA \\
 Editor & G. Kumfert A. Pothen
\end{tabular}
\vspace{2ex}

{
\begin{table}[h!]
\begin{center}
\renewcommand{\arraystretch}{1.4}
 \begin{tabular}{c |  c  c  c c  }
\hline
Levels &  2 & 3 & 4 & 5  \\  [1ex]
\hline
V-cycle & 10 & 11 & 12 & 13  \\  [1ex]
W-cycle & 10 & 10 & 10 & 10  \\ [1ex]
 \hline
CDOF & ~1411~ & ~~333~~ & ~~106~~ & ~~~37~~~ \\
 \hline
\end{tabular}
\end{center}
\caption{Example 2, $10\, 429$ DOF}
\label{tab:2}
\end{table}
}

\subsection*{Example 3}\nonumber


For the graph $K=(\mathcal{V},\mathcal{E})$ considered in this example we have $\vert \mathcal{V}\vert=16\,146$ and $\vert\mathcal{E}\vert=499\, 505$.
$89$ is the highest vertex degree whereas the minimum and the average are $24$ and $\approx 61,87$ respectively.

\vspace{2ex}
\begin{tabular}{l|l}
Description & STRUCTURE FROM NASA LANGLEY, ACCURACY PROBLEM ON Y-MP \\
Author & H. Simon \\
 Editor & H. Simon
\end{tabular}
\vspace{2ex}

{
\begin{table}[h!]
\begin{center}
\renewcommand{\arraystretch}{1.4}
 \begin{tabular}{c |  c  c  c c  }
\hline
Levels &  2 & 3 & 4 & 5   \\  [1ex]
\hline
V-cycle & 12 & 14 & 14 & 14  \\  [1ex]
W-cycle & 12 & 12 & 12 & 12  \\ [1ex]
\hline
CDOF & ~~643~~ & ~~143~~ & ~~~49~~~ & ~~~13~~~ \\
 \hline
\end{tabular}
\end{center}
\caption{Example 3, $16\,146$  DOF}
\label{tab:3}
\end{table}
}

\subsection*{Example 4}\nonumber


The graph presented in Example 4 has $67\,578$ nodes and $168\, 176$ edges.
Here the highest vertex degree is $53$, the minimum $1$ and the average $\approx 4,98$.

\vspace{2ex}
\begin{tabular}{l|l}
Description & DIMACS10 set: redistrict/ct2010 and ct2010a \\
Author & W. Zhao \\
 Editor & H. Meyerhenke
\end{tabular}
\vspace{2ex}

{
\begin{table}[h!]
\begin{center}
\renewcommand{\arraystretch}{1.4}
 \begin{tabular}{c |  c  c  c c c c  }
\hline
Levels &  2 & 3 & 4 & 5 & 6 & 7   \\  [1ex]
\hline
V-cycle & - & 13 & 14 & 16 & 17 & 17  \\  [1ex]
W-cycle & - & 10 & 10 & 10 & 10 & 10  \\ [1ex]
\hline
CDOF & ~22714~ & ~~5267~~ & ~~1497~~ & ~~~425~~~ & ~~~123~~~ & ~~~~39~~~~ \\
 \hline
\end{tabular}
\end{center}
\caption{Example 4, $67\,578$ DOF}
\label{tab:4}
\end{table}
}

\subsection*{Example 5}


In this Example the graph consists of $126\, 146$ nodes and  $161\,950$ edges. The maximum vertex degree is $7$, the minimum $1$ and 
the average $\approx 2,57$.

\vspace{2ex}
\begin{tabular}{l|l}
Description & Continental US road network (with $xy$ coordinates \\
Author & D. Gleich \\
 Editor & T. Davis
\end{tabular}
\vspace{2ex}

{
\begin{table}[h!]
\begin{center}
\renewcommand{\arraystretch}{1.4}
 \begin{tabular}{c |  c  c  c c  c c c c c}
\hline
Levels &  2 & 3 & 4 & 5 & 6 & 7 & 8 & 9 & 10 \\  [1ex]
\hline 
V-cycle & - & - & - & 14 & 16 & 18 & 19 & 20 & 20 \\  [1ex]
W-cycle & - & - & - & 7 & 7 & 7 & 7 & 7 & 7 \\ [1ex]
\hline 
CDOF & 57018 & 21671 & ~8703~ & ~3253~ & ~1201~ & ~~430~~ & ~~152~~ & ~~~51~~~ & ~~~15~~~\\
 \hline
\end{tabular}
\end{center}
\caption{Example 5, $126\, 146$ DOF}
\label{tab:5}
\end{table}
}

\subsection*{Example 6}


The graph in Example 6 has $143\, 437$ nodes,  $409\,593$ edges, maximum vertex degree $6$, minimum vertex degree $2$ and 
average vertex degree $\approx 5,71$.

\vspace{2ex}
\begin{tabular}{l|l}
Description & DIMACS10 set: walshaw/fe\_ocean \\
Author & F. Pellegrini \\
 Editor & C. Walshaw
\end{tabular}
\vspace{2ex}

{
\begin{table}[h!]
\begin{center}
\renewcommand{\arraystretch}{1.4}
 \begin{tabular}{c |  c  c  c c  c c}
\hline
Levels &  2 & 3 & 4 & 5 & 6 & 7 \\  [1ex]
\hline 
V-cycle & 12 & 13 & 15 & 15 & 15 & 15 \\  [1ex]
W-cycle & 12 & 13 & 13 & 13 & 13 & 13 \\ [1ex]
\hline 
CDOF & 71386 & 10026 & ~1927~ & ~~414~~ & ~~100~~ & ~~~31~~~\\
 \hline
\end{tabular}
\end{center}
\caption{Example 6: $143\, 437$ DOF}
\label{tab:6}
\end{table}
}

\subsection*{Example 7}


Example 7 considers a graph with $214\, 765$ nodes,  $1\,679\,018$ edges. The maximum, minimum and average vertex degrees are $40$, $4$ and 
$\approx 15,64$ respectively.

\vspace{2ex}
\begin{tabular}{l|l}
Description & DIMACS10 set: walshaw/m14b \\
Author & V. Kumar \\
 Editor & C. Walshaw
\end{tabular}
\vspace{2ex}

{
\begin{table}[h!]
\begin{center}
\renewcommand{\arraystretch}{1.4}
 \begin{tabular}{c |  c  c  c c  c c}
\hline
Levels &  2 & 3 & 4 & 5 & 6 & 7 \\  [1ex]
\hline 
V-cycle & - & 10 & 12 & 13 & 13 & 13 \\  [1ex]
W-cycle & - & 8 & 8 & 8 & 8 & 8 \\ [1ex]
\hline 
CDOF & 30238 & ~5634~ & ~1990~ & ~~288~~ & ~~69~~ & ~~~16~~~\\
 \hline
\end{tabular}
\end{center}
\caption{Example 7: $214\, 765$ DOF}
\label{tab:7}
\end{table}
}

\subsection*{Example 8}


The graph in Example 8 has $1\,048\, 572$ nodes and  $6\,891\,617$ edges. Its maximum vertex degree is $32$, its minimum $1$ and average $\approx 13,14$.

\vspace{2ex}
\begin{tabular}{l|l}
Description &  DIMACS10 set: random/rgg\_n\_2\_20\_s0 \\
Author & M. Holtgrewe P. Sanders C. Schulz \\
 Editor & C. Schulz
\end{tabular}
\vspace{2ex}

{
\begin{table}[h!]
\begin{center}
\renewcommand{\arraystretch}{1.4}
 \begin{tabular}{c |  c  c  c c  c c c c c}
\hline
Levels &  2 & 3 & 4 & 5 & 6 & 7 & 8 & 9 & 10 \\  [1ex]
\hline 
V-cycle & - & - & - & 16 & 18 & 19 & 19 & 19 & 19 \\  [1ex]
W-cycle & - & - & - & 11 & 11 & 11 & 11 & 11 & 11 \\ [1ex]
\hline 
CDOF & 160410 & ~48818~ & ~15966~ & ~~5249~~ & ~~1729~~ & ~~~586~~~ & ~~~208~~~ & ~~~~68~~~~ & ~~~~23~~~~\\
 \hline
\end{tabular}
\end{center}
\caption{Example 8: $1\,048\, 572$ DOF}
\label{tab:8}
\end{table}
}

\begin{remark}

We have implemented both the deflated CG method and the rank-$1$ update of the graph Laplacian matrix and as one would expect we have observed 
the same number of iterations in both cases. 
\end{remark}

\begin{remark}
In Tables~\ref{tab:5}--\ref{tab:8} we report the number of iterations only for the $\ell$-level methods with CDOF $\le 10\,000$ as solving exactly such big 
problems is time and memory consuming.
\end{remark}

\section{Conclusions}

The approach of the auxiliary space multigrid preconditioning based on ASCA as recently presented in~\cite{K-15} has been successfully extended to graph Laplacian matrices for general 
unstructured graphs. We have suggested a procedure and an algorithm to set up a hierarchy of coarse-grid operators that results in fast multigrid convergence. The 
operator complexity can be controlled by the size of the building components, see the discussion in Section~\ref{sec:complexity}.

We have numerically tested the proposed approach on several examples. 
As can be seen from the results in Tables~\ref{tab:1}--\ref{tab:8} the W-cycle converges uniformly in the number of levels, whereas
for Examples~\ref{tab:1}, \ref{tab:3}, \ref{tab:6}--\ref{tab:8} we observe also a uniformly 
convergent V-cycle. 

After all, it is still worth mentioning that the approach is 
purely algebraic and easy to be implemented from an algorithmic point of view. Furthermore, there are no limitations regarding the topology of the graphs. 

Future work will address weighted graph Laplacian systems. 

\bibliographystyle{unsrt}
\bibliography{thebib}

\end{document}